\documentclass[11pt]{article}

\usepackage[margin=1in]{geometry}
\usepackage{times}
\usepackage{verbatim}
\usepackage{amsmath,amssymb,amsthm,mathtools,bbm,bm}
\usepackage{xcolor}
\usepackage{graphicx}

\usepackage{subcaption}
\usepackage{algorithm}
\usepackage{algpseudocode}
\usepackage{hyperref}
\usepackage{float}
\usepackage[nameinlink]{cleveref}
\usepackage{booktabs}

\graphicspath{{figures/}}

\newtheorem{theorem}{Theorem}[section]
\newtheorem{lemma}[theorem]{Lemma}
\newtheorem{example}[theorem]{Example}
\newtheorem{proposition}[theorem]{Proposition}

\theoremstyle{definition}
\newtheorem{remark}[theorem]{Remark}
\newtheorem{definition}[theorem]{Definition}

\newcommand{\R}{\mathbb{R}}

\def\Bar{\overline}

\title{Solving Regularized Multifacility Location Problems with Unknown Number of Centers via Difference-of-Convex Optimization}

\author{{ W. Geremew}\footnote{ Department of Computer Information Systems
School of Business
Stockton University, USA (email: Wondi.Geremew@stockton.edu).},
{ V. S. T. Long\footnote{Faculty of Mathematics and Computer
Science, University of Science,
Ho Chi Minh City, Vietnam.}$^,$ \footnote{Vietnam National University, Ho Chi Minh City, Vietnam, (email: vstlong@hcmus.edu.vn).},
{ N. M. Nam}\footnote{Fariborz Maseeh Department of Mathematics and Statistics,
Portland State University, Portland, OR 97207, USA (email: mnn3@pdx.edu). }, { A. Solano-Herrera}\footnote{Fariborz Maseeh Department of Mathematics and Statistics, Portland State University, Portland, OR
97207, USA (email: asolano2@pdx.edu ).}}}
\date{}

\begin{document}
\maketitle

\begin{abstract} In this paper, we develop optimization methods for a new model of multifacility location problems  defined by a Minkowski gauge with Laplace-type regularization terms. The model is analyzed from both theoretical and numerical perspectives. In particular, we establish the existence of optimal solutions and study qualitative properties of global minimizers. By combining Nesterov’s smoothing technique with recent advances in difference-of-convex optimization, following the pioneering work of P. D. Tao and L. T. H. An and others, we propose efficient numerical algorithms for minimizing the objective function of this model. As an application, our approach provides an effective method for determining the number of centers in gauge-based multifacility location and clustering problems. Our results extend and complement recent developments by L. T. H. An in \cite{An2025}.

\end{abstract}


\section{Introduction}

The multifacility location problem aims to place several facilities in optimal locations to minimize the total distance to a set of demand points. These problems are closely related to clustering, which plays a central role in data analysis, signal processing, and machine learning. Given a finite collection of data points, the goal of a clustering problem is to identify a prescribed number of representative centers that capture the underlying structure of the data. Classical formulations, such as the $k$-means algorithm, typically rely on squared Euclidean distances and lead to nonconvex and nonsmooth optimization models. Their theoretical analysis and numerical solution have been extensively studied in the literature; see, for instance, \cite{Cuong2020, Cuong2023, CTWYOptim, LNSYkcenter, Martini2002, Mordukhovich2010, Nam2018}. 

Let $F \subset \mathbb{R}^d$ be a fixed compact convex set containing the origin in its interior and let $\rho_F\colon \R^d\to \R$ be its associated Minkowski gauge function defined by 
\begin{equation}\label{gauge}
\rho_F(x) = \inf\{\lambda \ge 0 \mid x \in \lambda F\}.    
\end{equation}
In this paper,  we propose a new model of multifacility location and  clustering based on generalized distances defined by Minkowski gauges and Laplace-type regularization terms. Consider a finite number of distinct data points $a_i$ for $i=1, \ldots, n$ in $\mathbb{R}^d$. We consider the optimization problem:
\begin{equation}\label{maxminn}
	\min_{X\in\mathbb{R}^{k\times d}} \Biggl\{ f(X) = \sum_{i=1}^{n} \min_{1 \le \ell \le k} \rho(x_\ell - a_i) + \frac{\lambda n}{2} \sum_{1 \le s < t \le k} \|x_s - x_t\|^2 \Biggr\},
\end{equation}
where  $\lambda>0$ is a regularization parameter and  $x_\ell$ denotes the $\ell$th row of the matrix variable $X$ for $\ell=1,\ldots,k$. Here and throughout the paper we use $\rho$ instead of $\rho_F$ for simplicity. The first term measures
the data fidelity through a generalized distance induced by $\rho$, while the
second term is a quadratic fusion penalty that promotes closeness among cluster
centers.
where $\lambda>0$ is a parameter. In contrast to prior work that uses
squared Euclidean losses for data fitting, our gauge optimization model allows us to solve problems of multifacility location and clustering, while  the regularization term  enhance stability, robustness, and interpretability. By controlling the  parameter $\lambda$, the regularization term encourages nearby centers to contract. Therefore, our new model also  allows us to determine the  effective number of cluster centers in multifacility location. The use of a regularization term in clustering can be found, for example, in \cite{An2025, Bouveyron,Gao, Huang, Mclachlan, Martini2002}.

The main difficulty in solving problem \eqref{maxminn} stems from the fact that its objective function is both nonsmooth and nonconvex. To address this challenge, we employ Nesterov’s smoothing technique (see  \cite{nesterov2018lectures, nesterov2005smooth}) to approximate the generalized distance $\rho(x_\ell-a_i)$, which constitutes the \emph{data fidelity term} of the objective function $f$, by a $C^{1,1}$ function. This smooth approximation also enables us to express $f$ as a difference of convex functions, making it amenable to difference-of-convex (DC) optimization methods. In particular, we apply the difference-of-convex algorithm (DCA) introduced by Tao and Souad \cite{Tao1986} (see also \cite{An2005, LeThi2018, LeThi2024, TA1, TA2} for further developments), together with its boosted variants from \cite{AragonArtacho2022, AragonArtacho2018}. Our work builds on the successful use of DC programming in clustering, facility location, and related nonsmooth models developed in \cite{An2025, An2007, AragonArtacho2022, AragonArtacho2018, Artacho2020, Nam2018, Nam2017}.

The main contributions of this paper include:
\begin{itemize}
    \item {\em A new model of gauge multifacility location and clustering:} We propose a new model of multifacility location and clustering based on generalized distances defined by Minkowski gauges and Laplace-type regularization terms as in \eqref{maxminn}. The model is applicable to multifacility location and, at the same time, extends classical clustering frameworks beyond standard Euclidean distances.
    \item {\em Qualitative analysis:} We establish the existence of solutions and investigate qualitative properties of both global and local minimizers of the proposed problem. In particular, we analyze the compactness of the set of global solutions as well as the stability of the optimal value function and the associated solution mapping under data perturbations.
    \item {\em Advanced DC programming via Nesterov's smoothing techniques:} We establish an exact DC decomposition of the nonsmooth objective function $f$ in \eqref{maxminn} that is well suited for advanced optimization techniques in DC programming, including the DCA and its boosted variants.
    \item {\em Simultaneous clustering and model selection:} We develop the \textit{LDCA-K} algorithm, which incorporates an adaptive cluster deletion mechanism. This approach allows the model to identify the optimal number of clusters during the optimization process. Our results extend and complement recent developments by L. T. H. An in \cite{An2025}.
\end{itemize}

The paper is structured as follows. Section \ref{section2} provides the qualitative analysis including existence results. Section \ref{section3} discusses stability properties. Sections \ref{section4} and \ref{section5} are devoted to the DC programming framework and the LDCA-K algorithm. Convergence analysis and numerical results are presented in Sections \ref{section6}, 
and~\ref{sec:section8}.

 Our analysis is conducted entirely in finite-dimensional Euclidean spaces and relies on tools from convex analysis, variational analysis, and matrix spectral theory; see \cite{boyd2004convex,mordukhovich2023easy,rockafellar1970convex,rockafellar2009variational} and the references therein. Throughout the paper, we consider the Euclidean space $\mathbb{R}^d$  equipped with the inner product $\langle \cdot, \cdot\rangle$ and the Euclidean norm $\|\cdot\|$. The closed and open Euclidean balls with center $a \in \mathbb{R}^n$ and radius $R > 0$ are respectively defined by $$B[a; R] = \{x \in \mathbb{R}^n \mid \|x-a\| \le R\}\quad\text{and}\quad B(a; R) = \{x \in \mathbb{R}^n \mid \|x-a\| < R\},$$ respectively. We also consider the space $\mathbb{R}^{k\times d}$ of all $k\times d$ real matrices equipped with the {\em Frobenius norm}.

\section{The Existence of Solutions}
\label{section2}
We begin this section by recalling some basic notions and definitions from convex analysis that will be employed throughout the paper.

\begin{definition} \label{def_v1} 
Let $\Omega$ be a nonempty subset of $\mathbb{R}^d$. 
\begin{enumerate}
    \item[{\rm (a)}] The \textit{distance function} associated with $\Omega$ is given by
    \begin{equation*}
        d(x; \Omega) = \inf \left\{ \|x - w\| \mid w \in \Omega \right\}, \quad x \in \mathbb{R}^d.
    \end{equation*}
    \item[{\rm (b)}] For any $x \in \mathbb{R}^d$, the \textit{Euclidean projection} of $x$ onto $\Omega$ is given as the set
    \begin{equation*}
        P(x; \Omega) = \{ \omega \in \Omega \mid \|x - \omega\| = d(x; \Omega) \}.
    \end{equation*}
\end{enumerate}
\end{definition}

The \textit{polar set} $F^\circ$ of a subset $F \subset \mathbb{R}^d$ is defined by
\begin{equation*}
    F^\circ = \{ v \in \mathbb{R}^d \mid \langle v, x \rangle \leq 1 \text{ for all } x \in F \}.
\end{equation*}
To simplify our estimates, we denote the maximum norms associated with the set $F$ and its polar $F^\circ$ as:
\begin{equation*}
    \|F\| = \sup \{ \|x\| \mid x \in F \} \quad \text{and} \quad \|F^\circ\| = \sup \{ \|v\| \mid v \in F^\circ \}.
\end{equation*}

We will use the following properties concerning the Minkowski function in our subsequent analysis. These results can be found in \cite[Proposition 2.1]{colombo2004subgradient},  \cite[Lemma 4.1]{longoptimletter}, and \cite[Theorem 6.14 and Proposition 6.18]{mordukhovich2023easy}.

\begin{lemma}\label{lemmarho0}
Let $\rho$ be the Minkowski function defined in \eqref{gauge}. Then the following properties hold:
	\begin{enumerate}
	\item [{\rm (a)}] $\rho(\alpha x) =\alpha\rho(x)$ for all $\alpha \geq0$,  $x\in \R^d$.
	\item [{\rm (b)}] $\rho(x_1+x_2)\leq \rho(x_1)+\rho(x_2)$ for all $x_1,x_2\in \R^d$.
	\item [{\rm (c)}] $\rho(x)=0\;\text{ if and only if }\; x=0.$
	\item [{\rm (d)}]  $\rho$ is $\|F^\circ\|$-Lipschitz continuous on $\R^d$.
	\item [{\rm (e)}] $x\in {\rm bd}(F)$ if and only if $\rho(x)=1,$ where ${\rm bd}(F)$ denotes the boundary of $F$.
	\item [{\rm (f)}] $\rho(x)=\max_{v\in F^\circ}\langle v,x\rangle$ for all $x\in \R^d.$
	\item [{\rm (g)}] $\frac{\rho(x)}{\|F^\circ\|}\leq \|x\|\leq \|F\|\rho(x)$ for all $x\in\R^d.$
\item [{\rm (h)}] $\rho (x)\leq \|F\|\|F^{\circ}\|\rho(-x)$ for all $x\in \R^d$.
	\end{enumerate}
\end{lemma}

\subsection{Global Optimal Solutions}
In this subsection, we establish the existence and compactness of the global optimal solution set for problem~\eqref{maxminn}, followed by illustrative examples.

\begin{theorem}\label{thm1}
The global optimal solution set $S$ of problem~\eqref{maxminn} is nonempty and compact.
\end{theorem}

\begin{proof}
By Lemma~\ref{lemmarho0}(d), the Minkowski gauge function $\rho$ is Lipschitz continuous on $\mathbb{R}^d$.  Thus, $f$ is continuous on $\mathbb{R}^{k\times d}$. We first show that $f$ is coercive, i.e.,
\begin{equation}\label{eq1}
\|X\|_F \to \infty \quad \Longrightarrow \quad f(X) \to \infty,
\end{equation}
where $\|\cdot\|_F$ denotes the Frobenius norm on $\mathbb{R}^{k\times d}$.
Let $\{X_m\}_m$ be a sequence in $\mathbb{R}^{k\times d}$ such that 
$$\|X_m\|_F \to \infty\text{\; as }\;m \to \infty,$$ where $X_m = (x_{1,m}, \dots, x_{k,m}) \in \mathbb{R}^{k \times d}$. Then there exists an index $t_0 \in \{1, \dots, k\}$ such that $$\|x_{t_0,m}\| \to \infty\;\text{ as }\;m \to \infty.$$ 
We now distinguish two cases:

\medskip
\noindent\emph{Case A: There exists $s_0 \neq t_0$ such that the sequence $\{\|x_{t_0,m} - x_{s_0,m}\|\}_m$ is unbounded.} 
By passing to a subsequence if necessary, we may assume that $$\|x_{t_0,m} - x_{s_0,m}\| \to \infty\;\text{ as }\;m \to \infty.$$ 
In this case, we consider the fusion term of $f$. Let $(s', t')$ be a pair of indices such that $\{s', t'\} = \{s_0, t_0\}$ with $s' < t'$. Then we have
\[
\frac{\lambda n}{2} \sum_{1 \le s < t \le k} \|x_{s,m} - x_{t,m}\|^2 
\ge \frac{\lambda n}{2} \|x_{s',m} - x_{t',m}\|^2 
= \frac{\lambda n}{2} \|x_{t_0,m} - x_{s_0,m}\|^2 \to \infty,
\]
which implies that $$f(X_m) \to \infty\;\text{ as }\;m \to \infty.$$

\medskip
\noindent\emph{Case B: For every $\ell = 1, \dots, k$, the sequence $\{\|x_{t_0,m} - x_{\ell,m}\|\}_m$ is bounded.} 
Since $\|x_{t_0,m}\| \to \infty$, it follows from the triangle inequality that for all $\ell = 1, \dots, k$,
\[
\|x_{\ell,m}\| \ge \|x_{t_0,m}\| - \|x_{t_0,m} - x_{\ell,m}\| \to \infty \;\text{ as }\; m \to \infty.
\]
Fix any index $i_0 \in \{1, \dots, n\}$. For every $\ell = 1, \dots, k$, we have$$\|x_{\ell,m} - a_{i_0}\| \ge \|x_{\ell,m}\| - \|a_{i_0}\| \to \infty \quad \text{as } m \to \infty,$$because $\|a_{i_0}\|$ is a fixed constant. 
Applying the lower bound property of the gauge function from Lemma~\ref{lemmarho0}(g), for every $\ell = 1, \dots, k$, we have
\[
\rho(x_{\ell,m} - a_{i_0}) \ge \frac{1}{\|F\|} \|x_{\ell,m} - a_{i_0}\| \to \infty \;\text{ as }\; m \to \infty.
\]
Hence, the minimum over $\ell$ also diverges, i.e.,
\[
\min_{1 \le \ell \le k} \rho(x_{\ell,m} - a_{i_0}) \to \infty \;\text{ as }\; m \to \infty.
\]
Consequently, the data-fidelity term satisfies
\[
\sum_{i=1}^n \min_{1 \le \ell \le k} \rho(x_{\ell,m} - a_i) \ge \min_{1 \le \ell \le k} \rho(x_{\ell,m} - a_{i_0}) \to \infty \;\text{ as }\; m \to \infty.
\]
As the fusion term is non-negative, we again conclude that $f(X_m) \to \infty$.

\medskip
Combining Cases A and B, we conclude that $f$ is coercive. Since $f$ is continuous and coercive on the finite-dimensional space $\mathbb{R}^{k\times d}$, the set of global optimal solutions $S$ is nonempty and compact (see, e.g., \cite[Theorem 1.9]{rockafellar2009variational}), which completes the proof.
\end{proof}

The following example shows that in the case $\lambda=0$, one cannot use coercivity
arguments to prove the existence and compactness of the global optimal
solution set to problem \eqref{maxminn}. 

\begin{example}{\rm 
Consider the simplest one-dimensional setting $d=1$ and let $F = [-1,1]\subset\mathbb{R}.$
Then the associated Minkowski gauge is the absolute value,
\[
\rho(x)=|x|,\qquad x\in\mathbb{R}.
\]
Take two data points $a_1=0,\; a_2=1,$
and $k=2$ centers. If we remove the penalty term (i.e.\ set $\lambda=0$),
problem~\eqref{maxminn} reduces to
\[
\min_{(x_1,x_2)\in\mathbb{R}^2}\Big\{f(x_1,x_2)
=\sum_{i=1}^2 \min_{\ell=1,2} |x_\ell-a_i|\Big\}.\]
Take the sequence
\[
X_m = (x_{1,m},x_{2,m}) = (0,m),\qquad m\in\mathbb{N}.
\]
Then the Frobenius norm of $X_m$ satisfies
\[
\|X_m\|_F = \sqrt{(0)^2 + m^2} = m \;\to\; \infty
\;\text{ as }\; m\to\infty.
\]
However, for each $m\geq 2$ we have
\[
f(X_m)
=
\min\big(|0-0|,|m-0|\big)
+
\min\big(|0-1|,|m-1|\big)
=
0 + 1
= 1,
\]
since the first center $x_{1,m}=0$ fits $a_1=0$ perfectly and remains the closest
center to $a_2=1$ as soon as $m$ is large.
Thus we obtain a sequence $\{X_m\}$ with
\[
\|X_m\|_F\to\infty
\quad\text{but}\quad
f(X_m)=1\ \text{ for all } m\geq 2,
\]
which shows that $f$ is \emph{not} coercive when $\lambda=0$.}
\end{example}

The next example provides a simple one-dimensional illustration of Theorem~\ref{thm1}. 

\begin{example}\label{ex:lambda1-k3}
{\rm Let $\{a_1, a_2\} = \{0, 1\}$, $k=3$ centers, $F = [-1, 1]$, and $\lambda = 1$. Then
problem~\eqref{maxminn} takes the form
\[
\min_{(x_1,x_2,x_3)\in\R^3}
\Big\{
f(x_1,x_2,x_3)
=\sum_{i=1}^2\min_{1\le\ell\le 3}|x_\ell-a_i|
+\frac{\lambda n}{2}\sum_{1\le s<t\le 3}|x_s-x_t|^2
\Big\},
\]
that is,
\[
\min_{(x_1,x_2,x_3)\in\R^3}
\Big\{
\min(|x_1|,|x_2|,|x_3|)
+\min(|x_1-1|,|x_2-1|,|x_3-1|)
+\sum_{1\le s<t\le 3}|x_s-x_t|^2
\Big\}.
\]

\medskip\noindent
\emph{Step 1: Reduction to the interval $[0,1]$ and ordered centers.}
Since the objective is symmetric in $(x_1,x_2,x_3)$, we may assume without loss
of generality that
\[
x_1\le x_2\le x_3.
\]
If $x_1<0$, we can move $x_1$ to $0$, then $\min(|x_1|,|x_2|,|x_3|)$ and all
pairwise distances involving $x_1$ decrease (or stay the same), while
$\min(|x_1-1|,|x_2-1|,|x_3-1|)$ does not increase. A similar argument shows
that if $x_3>1$ we can move $x_3$ to $1$ without increasing the objective.
Hence any global minimizer must satisfy
\begin{equation}\label{eq2}
0\le x_1\le x_2\le x_3\le 1.
\end{equation}

\medskip\noindent
\emph{Step 2: Explicit form of the objective on $[0,1]$.}
For $0\le x_1\le x_2\le x_3\le 1$, we have
\[
\min(|x_1|,|x_2|,|x_3|)=x_1\quad\text{and}\quad
\min(|x_1-1|,|x_2-1|,|x_3-1|)=1-x_3,
\]
and thus
\[
\sum_{i=1}^2\min_{1\le\ell\le 3}|x_\ell-a_i|
=
x_1+(1-x_3).
\]
The penalty term becomes
\[
\sum_{1\le s<t\le 3}|x_s-x_t|^2
=
(x_2-x_1)^2+(x_3-x_1)^2+(x_3-x_2)^2.
\]
Therefore, for any $(x_1,x_2,x_3)$ satisfying \eqref{eq2}, we have
\begin{equation}\label{eq:ex-k3-f}
f(x_1,x_2,x_3)
=
x_1+(1-x_3)
+(x_2-x_1)^2+(x_3-x_1)^2+(x_3-x_2)^2.
\end{equation}

\medskip\noindent
\emph{Step 3: Minimization in terms of the gaps between centers.}
Letting
\[
u=x_2-x_1\ge0\quad\text{and}\quad v=x_3-x_2\ge0,
\]
we obtain
\[
x_2=x_1+u\quad\text{and}\quad x_3=x_2+v=x_1+u+v.
\]
Then the constraint $x_3\le 1$ becomes
\begin{equation*}\label{eq:ex-k3-ab-dom}
x_1+u+v\le1,\qquad x_1\ge0,\quad u\ge0,\quad v\ge0.
\end{equation*}
Substituting into \eqref{eq:ex-k3-f} yields
\[
\begin{aligned}
f(x_1,x_2,x_3)
&=x_1+\bigl(1-(x_1+u+v)\bigr)
+u^2+(u+v)^2+v^2\\[0.2em]
&=1-u-v +\bigl(u^2+(u+v)^2+v^2\bigr)\\[0.2em]
&=1-u-v +2(u^2+uv+v^2).
\end{aligned}
\]
Thus $f$ depends only on $(u,v)$, and we can write
\[
g(u,v)=
1-u-v+2(u^2+uv+v^2),\qquad u,v\ge0,\ u+v\le1.
\]

\medskip\noindent
\emph{Step 4: Minimizing $g(u,v)$.}
The function $g$ is strictly convex quadratic in $(u,v)$, since its Hessian matrix
$$H_g(u,v) = \begin{pmatrix} 4 & 2 \\ 2 & 4 \end{pmatrix}$$
is constant and positive definite. Hence $g$ has a unique global minimizer in
$\R^2$, and any critical point in the interior of the feasible region
$\{(u,v)\mid u,v\ge0,\ u+v\le1\}$ is the global minimizer.
Solving $\nabla g(u,v)=0$ gives
\[
\begin{cases}
\displaystyle \frac{\partial g}{\partial u}
=-1+4u+2v=0\\[0.4em]
\displaystyle \frac{\partial g}{\partial v}
=-1+2u+4v=0,
\end{cases}
\]
whose unique solution is
\[
u^\star=v^\star=\frac{1}{6}.
\]
Since $u^\star+v^\star=\frac{1}{3}\le1$, this point lies in the interior of
the feasible region. Substituting into $g$ yields
\[
g(u^\star,v^\star)
=
1-\frac{1}{6}-\frac{1}{6}
+2\Bigl(\Bigl(\frac{1}{6}\Bigr)^2
+\frac{1}{6}\cdot\frac{1}{6}
+\Bigl(\frac{1}{6}\Bigr)^2\Bigr)
=
\frac{5}{6}.
\]

\medskip\noindent
\emph{Step 5: Description of the optimal centers.}
We have shown that any global minimizer must satisfy
\[
u=x_2-x_1=\frac{1}{6},\qquad
v=x_3-x_2=\frac{1}{6},
\]
that is,
\[
x_2=x_1+\frac{1}{6},\qquad
x_3=x_1+\frac{1}{3}.
\]
The constraints $0\le x_1\le x_2\le x_3\le1$ become
\[
0\le x_1\le x_1+\frac{1}{6}\le x_1+\frac{1}{3}\le1,
\]
which implies
\[
0\le x_1\le\frac{2}{3}.
\]
Thus the set of global minimizers is the line segment
\[
S
=
\Bigl\{
(x_1,x_2,x_3)\in\R^3\ \Big|\ 
0\le x_1\le\frac{2}{3},\ \ 
x_2=x_1+\frac{1}{6},\ \ 
x_3=x_1+\frac{1}{3}
\Bigr\},
\]
and the minimal value is $f=\frac{5}{6}$. The solution set $S$ is
clearly compact, in agreement with Theorem~\ref{thm1}.

\medskip
For comparison, when $\lambda=0$ (and $k^{\star} = 3$) the optimal value is $0$, attained
for example by $(x_1,x_2,x_3)=(0,0,1)$ or $(0,1,1)$, since each demand point
can be matched exactly by at least one center. With $\lambda=1$, the penalty
term enforces repulsion among the centers and yields a nontrivial trade-off
between approximation quality and separation: in an optimal solution, the
three centers are equally spaced with gap $\frac{1}{6}$ between consecutive
centers and total span $\frac{1}{3}$.}
\end{example}

Following \cite{Cuong2020,LNSYkcenter}, we associate with each
$X=(x_1,\ldots,x_k)\in\R^{k\times d}$ a family of subsets
$A^1,\ldots,A^k$ of the demand set $A$ as follows.  
Set $A^0=\emptyset$ and for $\ell=1,\ldots,k$, define
\[
A^\ell
=
\left\{a_i\in A\setminus\Big(\bigcup_{q=0}^{\ell-1}A^q\Big)\;\Big|\;
\rho(x_\ell-a_i)
=
\min_{r=1,\ldots,k}\rho(x_r-a_i)
\right\}.
\]
We say that the family $\{A^1,\ldots,A^k\}$ constructed in this way is the
\emph{natural clustering} associated with $X$.

\begin{definition}\label{def1}
For $\ell=1,\dots,k$, the component $x_\ell$ of $X=(x_1,\ldots,x_k)\in \R^{k\times d}$ is said to be \textit{attractive} with respect to $A$ if the following set
\begin{equation}\label{def:A_l}
A_\ell(X)=\Big\{a_i\in A \;\Big|\;
\rho(x_\ell-a_i)=\min_{1\le r\le k}\rho(x_r-a_i)
\Big\}    
\end{equation}
is nonempty.
\end{definition}
Clearly, we have
$$A^\ell=A_\ell(X)\Big\backslash\left(\bigcup_{q=1}^{\ell-1}A^q\right).$$

The next proposition clarifies the role of duplicated centers in problem~\eqref{maxminn}.

\begin{proposition}\label{prop:natural-weaker}
Consider problem~\eqref{maxminn}. Let 
$X=(x_1,\dots,x_k)\in \R^{k\times d}$
and let $\widehat x_1,\dots,\widehat x_q$ be the distinct values among 
$\{x_1,\dots,x_k\}$. 
Then the following properties hold:
\begin{itemize}
\item[(i)] For every $i\in\{1,\dots,n\}$, there exists at least one 
$s\in\{1,\dots,q\}$ such that $a_i\in A_s(X)$, where
\[
 C_s=\{\ell\in\{1,\dots,k\}\mid x_\ell=\widehat x_s\}\quad\text{and}\quad
A_s(X)=\bigcup_{\ell\in C_s} A_\ell(X).
\]
In other words,
\[
A=\bigcup_{s=1}^q A_s(X).
\]

\item[(ii)] If we discard empty sets among $\{A_s(X)\}_{s=1}^q$, the remaining
family $\{A_s(X)\}_{s\in J}$ with
\[
J=\{s\in\{1,\dots,q\}\mid A_s(X)\neq\emptyset\}
\]
forms a natural clustering of $A$ induced by the distinct center locations
$\{\widehat x_s\}_{s\in J}$.
In particular, the number of \emph{effective} clusters satisfies
$|J|\le\min\{k,n\}$. 
\end{itemize}
\end{proposition}

\begin{proof}
(i) Fix any $i\in\{1,\dots,n\}$. Then, there exists at least one index 
$\ell\in\{1,\dots,k\}$ such that
\[
\rho(x_\ell-a_i)=\min_{r=1,\dots,k}\rho(x_r-a_i),
\]
that is, $a_i\in A_\ell(X)$ by Definition \ref{def:A_l}.
On the other hand, each index $\ell\in\{1,\dots,k\}$ belongs to exactly one
class $C_s$ for some $s\in\{1,\dots,q\}$, because
$\widehat x_1,\dots,\widehat x_q$ are precisely the distinct values in
$\{x_1,\dots,x_k\}$. Let $s$ be such that $\ell\in C_s$. Then, by the
definition of $A_s(X)$, we have $A_\ell(X)\subset A_s(X)$. Hence
$a_i\in A_s(X)$.
Since $i$ was arbitrary, every $a_i$ belongs to at least one of the sets
$A_s(X)$, and therefore
\[
A=\bigcup_{s=1}^q A_s(X).
\]
This proves~\textup{(i)}.

\medskip\noindent
(ii) By the definition of $J$, each $A_s(X)$ with $s\in J$ is nonempty. From part~\textup{(i)}, we know that
\[
A=\bigcup_{s=1}^q A_s(X)=\bigcup_{s\in J}A_s(X),
\]
since the sets $A_s(X)$ with $s\notin J$ are empty.
Moreover, if $s\neq s'$, then $\widehat x_s\neq\widehat x_{s'}$, so the
indices in $C_s$ and $C_{s'}$ correspond to different center locations.
 Thus, the
family $\{A_s(X)\}_{s\in I}$ provides a partition of $A$ associated with the
distinct centers $\{\widehat x_s\}_{s\in J}$, and hence defines a natural
clustering of $A$. 

Finally, by construction we have $q\le k$, so $|J|\le q\le k$. On the other
hand, the sets $A_s(X)$ for $s\in J$ are nonempty and are all subsets of $A$,
which implies $|J|\le n$. Combining these inequalities yields
\[
|J|\le\min\{k,n\},
\]
which completes the proof.
\end{proof}

We illustrate Proposition \ref{prop:natural-weaker} by the following example.

\begin{example}\label{ex2:lambda1-k3}
{\rm In the setting of Example \ref{ex:lambda1-k3}, we illustrate Proposition~\ref{prop:natural-weaker} by considering
\[
X=(x_1,x_2,x_3)=(0,\tfrac{1}{2},\tfrac{1}{2}).
\]
The three centers take only two distinct values $
\widehat x_1=0$ and $\widehat x_2=\tfrac{1}{2}.$
Thus, we obtain $q=2$, 
\[
C_1=\{\ell\in\{1,2,3\}\mid x_\ell=\widehat x_1\}=\{1\}\quad\text{and}\quad
C_2=\{\ell\in\{1,2,3\}\mid x_\ell=\widehat x_2\}=\{2,3\}.
\]
We compute $A_\ell[X]$ for $\ell=1,2,3$ explicitly:
\begin{itemize}
\item For $a_1=0$, we have
\[
|x_1-0|=|0-0|=0,\quad
|x_2-0|=\Bigl|\tfrac{1}{2}-0\Bigr|=\tfrac{1}{2},\quad
|x_3-0|=\Bigl|\tfrac{1}{2}-0\Bigr|=\tfrac{1}{2},
\]
 hence$a_1\in A_1(X)$ and $a_1\notin A_2(X)\cup A_3(X).$
\item For $a_2=1$, we have
\[
|x_1-1|=|0-1|=1,\quad
|x_2-1|=\Bigl|\tfrac{1}{2}-1\Bigr|=\tfrac{1}{2},\quad
|x_3-1|=\tfrac{1}{2},
\]
so the minimum is attained at both $\ell=2$ and $\ell=3$, and hence
\[
a_2\in A_2(X)\cap A_3(X),\qquad a_2\notin A_1(X).
\]
\end{itemize}
Thus
\[
A_1(X)=\{a_1\},\qquad A_2(X)=\{a_2\},\qquad A_3(X)=\{a_2\}.
\]
Now we regroup according to the distinct locations $\widehat x_1,\widehat x_2$
as in Proposition~\ref{prop:natural-weaker}:
\[
A_1(X)=\bigcup_{\ell\in C_1}A_\ell(X)=A_1(X)=\{a_1\}\quad\text{and}\quad
A_2(X)=\bigcup_{\ell\in C_2}A_\ell(X)=A_2(X)\cup A_3(X)=\{a_2\}.
\]
Hence
\[
A=A_1(X)\cup A_2(X)=\{a_1,a_2\},
\]
while both $A_1(X)$ and $A_2(X)$ are nonempty. In the notation of
Proposition~\ref{prop:natural-weaker}, we have
\[
J=\{s\in\{1,2\}\mid A_s(X)\neq\emptyset\}=\{1,2\},
\]
so the number of effective clusters is
\[
|J|=2\le\min\{k,n\}=\min\{3,2\}=2.
\]
This example shows how several centers (here $x_2$ and $x_3$) can coincide at
the same location $\widehat x_2=\tfrac{1}{2}$ while still giving rise to only
one effective cluster $A_2(X)=\{a_2\}$. The family
$\{A_s(X)\}_{s\in I}=\bigl\{\{a_1\},\{a_2\}\bigr\}$ is a natural clustering of
$A$ associated with the distinct center locations $\widehat x_1=0$ and
$\widehat x_2=\tfrac{1}{2}$.}
\end{example}

\subsection{Local Optimal Solutions}

This subsection is devoted to studying local optimal solutions of problem~\eqref{maxminn}.

For $X=(x_1,\dots,x_k)\in\R^{k\times d}$ and $\ell\in\{1,\dots,k\}$, set
\[
I_\ell(X)=
\{i\in\{1,\dots,n\}\mid a_i\in A_\ell(X)\},
\]
where $A_\ell(X)$ is defined as in \eqref{def:A_l}.
Then we have $$A_\ell(X)=\{a_i\mid i\in I_\ell(X)\}.$$

\begin{definition}\label{deflocal}
An element $\Bar X=(\bar x_1,\dots,\bar x_k)\in\R^{k\times d}$ is called a
\emph{local optimal solution} of problem~\eqref{maxminn} if there exists
$\varepsilon>0$ such that
\[
f(\Bar X)\;\le\;f(X)
\quad
\text{for all }X=(x_1,\dots,x_k)\in\R^{k\times d}
\text{ with }\|x_\ell-\bar x_\ell\|<\varepsilon
\text{ for all }\ell.
\]
The set of all local optimal solutions will be denoted by $S_{loc}$.
\end{definition}

As in~\cite{LNTV}, it is convenient to encode the assignment of the cluster of each demand point in terms of an index set.

\begin{definition}\label{def2}
For $X=(x_1,\dots,x_k)\in\R^{k\times d}$ and $i\in\{1,\dots,n\}$ we set
\[
L_i(X)=\bigl\{\ell\in\{1,\dots,k\}\mid a_i\in A_\ell(X)\bigr\}.
\]
Then we obtain
\[
\ell\in L_i(X)
\quad\Longleftrightarrow\quad
\rho(x_\ell-a_i)
=
\min_{r=1,\dots,k}\rho(x_r-a_i).
\]
\end{definition}
Thus $L_i(X)$ is the set of indices of all nearest centers to $a_i$. In particular, we have
\begin{equation}\label{indexI}
i\in I_\ell(X)
\quad\Longleftrightarrow\quad
\ell\in L_i(X).    
\end{equation}

The next lemma shows that, when all nearest centers are uniquely defined, the clustering is locally stable under small perturbations of the centers.

\begin{lemma}\label{lem:Li-stable}
Let $\Bar X=(\bar x_1,\dots,\bar x_k)\in\R^{k\times d}$ be such that
$L_i(\Bar X)$ is a singleton for every $i=1,\dots,n$. Then there exists
$\varepsilon>0$ such that for every
$X=(x_1,\dots,x_k)\in\R^{k\times d}$ satisfying
$\|x_\ell-\bar x_\ell\|<\varepsilon$ for all $\ell=1,\dots,k$, we have:
\begin{itemize}
    \item [\rm{(a)}] $L_i(X) = L_i(\Bar X)\quad\text{for all } i=1,\dots,n.$
    \item [\rm{(b)}] 
$I_\ell(X) = I_\ell(\Bar X)\quad\text{and}\quad A_\ell(X)=A_\ell(\Bar X)
\quad\text{for all } \ell=1,\dots,k.$
\end{itemize}
\end{lemma}

\begin{proof}
(a) For each $i\in\{1,\dots,n\}$, let $\ell(i)$ be the unique element of
$L_i(\Bar X)$. By Definition \ref{def1} and \ref{def2}, we have
\[
\rho(\bar x_{\ell(i)}-a_i)
<
\rho(\bar x_\ell-a_i)
\qquad\text{for all }\ell\in\{1,\dots,k\}\setminus\{\ell(i)\}.
\]
Using the continuity of $\rho$, for every $i\in \{1,\dots,n\}$ we can choose
$\varepsilon_i>0$ such that
\[
\rho(x_{\ell(i)}-a_i)
<
\rho(x_\ell-a_i)
\qquad\text{for all }\ell\in\{1,\dots,k\}\setminus\{\ell(i)\},
\]
whenever $X=(x_1,\dots,x_k)\in \R^{k\times d}$ satisfies
$\|x_\ell-\bar x_\ell\|<\varepsilon_i$ for all $\ell=1,\dots,k$.
For such $X$, we then have $$L_i(X)=\{\ell(i)\}=L_i(\Bar X).$$
Now take $\varepsilon=\min\{\varepsilon_1,\dots,\varepsilon_n\}$ and fix any
$X=(x_1,\dots,x_k)$ with $\|x_\ell-\bar x_\ell\|<\varepsilon$ for all $\ell$.
From the previous paragraph we obtain $L_i(X)=L_i(\Bar X)$ for all $i$,
which proves the first assertion.

(b) Fix any $\ell\in\{1,\dots,k\}$. Then it follows from \eqref{indexI} that
\[
i\in I_\ell(\Bar X)
\quad\Longleftrightarrow\quad
\ell\in L_i(\Bar X)
\quad\text{and}\quad
i\in I_\ell(X)
\quad\Longleftrightarrow\quad
\ell\in L_i(X).
\]
Since $L_i(X)=L_i(\Bar X)$ for all $i$, the two index sets coincide:
$I_\ell(X)=I_\ell(\Bar X)$. Hence we also obtain
$$A_\ell(X)=\{a_i\mid i\in I_\ell(X)\}=\{a_i\mid i\in I_\ell(\Bar X)\}
=A_\ell(\Bar X).$$ The proof is complete.
\end{proof}

We now derive a necessary optimality condition for local optimal solutions of
problem~\eqref{maxminn}.

\begin{theorem}\label{thm:local}
Let $\Bar X=(\bar x_1,\dots,\bar x_k)\in\mathbb{R}^{k\times d}$ be such that $L_i(\Bar X)$ is a singleton for every $i=1,\dots,n$. If $\Bar X$ is a local optimal solution of problem~\eqref{maxminn}, then for each $\ell\in\{1,\dots,k\}$, the point $\bar x_\ell$ is a global minimizer of the problem
\begin{equation}\label{eq:phi-ell}
\min_{x\in\mathbb{R}^d} \Biggl\{
\phi_\ell(x) = \sum_{i\in I_\ell(\Bar X)}\rho(x-a_i) + \frac{\lambda n}{2} \sum_{j=1,\; j \neq \ell}^k \|x-\bar x_j\|^2
\Biggr\}.
\end{equation}
\end{theorem}

\begin{proof}
By Lemma~\ref{lem:Li-stable}, the assumption that $L_i(\Bar X)$ is a singleton
for every $i$ implies the following stability property, i.e., there exists
$\varepsilon>0$ such that, for all $X=(x_1,\dots,x_k)\in\R^{k\times d}$ with
\begin{equation}\label{eq4}
\|x_\ell-\bar x_\ell\|<\varepsilon\qquad\text{for all }\ell=1,\dots,k,
\end{equation}
we have
\begin{equation}\label{eq5}
L_i(X)=L_i(\Bar X)\qquad\text{for all }i=1,\dots,n.
\end{equation}
In particular, in the neighbourhood defined by~\eqref{eq4}, we have
\begin{equation}\label{key15}
    I_\ell(X)=I_\ell(\Bar X)
\quad\text{and}\quad
A_\ell(X)=A_\ell(\Bar X)
\qquad\text{for all }\ell=1,\dots,k.
\end{equation}
Suppose now that $\Bar X$ is a local optimal solution of~\eqref{maxminn}. 
By Definition~\ref{deflocal}, there exists $\delta>0$ such that
\[
f(\Bar X)\;\le\;f(X)
\quad
\text{for all }X=(x_1,\dots,x_k)\in\R^{k\times d}
\text{ with }\|x_\ell-\bar x_\ell\|<\delta
\text{ for all }\ell.
\]
Let $\varepsilon>0$ be as in~\eqref{eq4}, and set $
\gamma = \min\{\varepsilon,\delta\}.$
 For a fixed index $\ell$, take any $x\in\mathbb{R}^d$ such that $\|x-\bar x_\ell\| < \gamma$ and define 
 \[
\widehat x_\ell = x
\quad\text{and}\quad
\widehat x_r = \bar x_r \ \text{ for } r\neq\ell.
\]
Then $\widehat X$ satisfies \eqref{eq4}, and by local optimality we have
\begin{equation}\label{eq6}
f(\Bar X)\le f(\widehat X).
\end{equation}
To compare these two values, write
\begin{equation}\label{fDC}
f(X)=D(X)+C(X),    
\end{equation}
where
\[
D(X)=\sum_{i=1}^n\min_{1\le \ell\le k}\rho(x_\ell-a_i)
\quad\text{and}\quad
C(X)=\frac{\lambda n}{2}\sum_{1\le s<t\le k}\|x_s-x_t\|^2.
\]
 By~\eqref{eq5} and \eqref{key15}, we have
\begin{equation*}\label{eq:D-bar}
D(\Bar X)
=
\sum_{i\in I_\ell(\Bar X)}\rho(\bar x_\ell-a_i)
+
\sum_{i\notin I_\ell(\Bar X)}\min_{1\le r\le k}\rho(\bar x_r-a_i),
\end{equation*}
and
\begin{equation*}\label{eq:D-hat}
\begin{aligned}
     D(\widehat X)
     &=
\sum_{i\in I_\ell(\widehat X)}\rho(x-a_i)
+
\sum_{i\notin I_\ell(\widehat X)}\min_{1\le r\le k}\rho(\bar x_r-a_i) \\
     &=
\sum_{i\in I_\ell(\Bar X)}\rho(x-a_i)
+
\sum_{i\notin I_\ell(\Bar X)}\min_{1\le r\le k}\rho(\bar x_r-a_i),
\end{aligned}
\end{equation*}
which implies that
\begin{equation}\label{eq:D-diff}
D(\widehat X) - D(\Bar X) = \sum_{i\in I_\ell(\Bar X)} \rho(x-a_i) - \sum_{i\in I_\ell(\Bar X)} \rho(\bar x_\ell-a_i).
\end{equation}

For the fusion term, we can group as follows:
\begin{equation}\label{eq:C-bar-new}
C(\Bar{X}) = \frac{\lambda n}{2} \left( \sum_{j=1,\; j \neq \ell}^k \|\bar{x}_\ell - \bar{x}_j\|^2 + \sum_{1 \le s < t \le k,\; s,t \neq \ell} \|\bar{x}_s - \bar{x}_t\|^2 \right).
\end{equation}
Similarly for $\widehat{X}$, we have
\begin{equation}\label{eq:C-hat-new}
C(\widehat{X}) = \frac{\lambda n}{2} \left( \sum_{j=1,\; j \neq \ell}^k \|x - \bar{x}_j\|^2 + \sum_{1 \le s < t \le k,\; s,t \neq \ell} \|\bar{x}_s - \bar{x}_t\|^2 \right).
\end{equation}
Subtracting \eqref{eq:C-bar-new} from \eqref{eq:C-hat-new} yields
\begin{equation*}\label{eq:C-diff}
C(\widehat{X}) - C(\Bar{X}) = \frac{\lambda n}{2} \sum_{j=1,\; j \neq \ell}^k \left( \|x - \bar{x}_j\|^2 - \|\bar{x}_\ell - \bar{x}_j\|^2 \right).
\end{equation*}
Combining this with \eqref{fDC} and
\eqref{eq:D-diff}, we obtain
\begin{equation}
f(\widehat X) - f(\Bar X) = \phi_\ell(x) - \phi_\ell(\bar x_\ell),
\end{equation}
where $\phi_\ell$ is defined in \eqref{eq:phi-ell}. It follows from \eqref{eq6} that
$\phi_\ell(\bar x_\ell) \le \phi_\ell(x)$ for all $x$ in a neighborhood of $\bar x_\ell$. 
Since $\phi_\ell$ is convex,
every local minimizer is a global minimizer. Hence
$\bar x_\ell$ is a global solution of~\eqref{eq:phi-ell} for every $\ell$,
which proves the theorem.
\end{proof}

\begin{remark}
When $\lambda=0$, the penalty term $C(X)$ vanishes and the functions
$\phi_\ell$ in~\eqref{eq:phi-ell} reduce to
\[
\phi_\ell(x)=\sum_{i\in I_\ell(\Bar X)}\rho(x-a_i),
\]
so Theorem~\ref{thm:local} recovers the necessary optimality condition for the
generalized multi-source Weber problem obtained in~\cite[Theorem~3.11]{LNTV}.
For $\lambda>0$, we do not claim that the converse implication holds in general.
\end{remark}

The following example shows that the condition in Theorem~\ref{thm:local} is genuinely necessary.

\begin{example}\label{ex:local-thm-lambda-1/2}
{\rm
Consider two data points
$a_1=0,$ $a_2=2,$ $k=2$ centers and $\rho(x)=|x|$ for all $x\in\R$. For $X=(x_1,x_2)\in\R^{2\times 1}$, the
objective function in~\eqref{maxminn} with $\lambda=\tfrac12$ reads
\[
f(x_1,x_2)
=
\sum_{i=1}^2 \min_{\ell=1,2}|x_\ell-a_i|
+\frac12\,|x_1-x_2|^2.
\]
For $\Bar X=(\bar x_1,\bar x_2)=(0,2)$,
we have
\[\min_{\ell=1,2}|\bar x_\ell-a_1|=\min\{|0-0|,\;|2-0|\}=0\]
and
\[\min_{\ell=1,2}|\bar x_\ell-a_2|=\min\{|0-2|,\;|2-2|\}=0.\]
Thus we obtain
\[f(\Bar X)=0+0+\frac12\,|0-2|^2=2,\]
\[L_1(\Bar X)=\{1\}\quad\text{and}\quad L_2(\Bar X)=\{2\},\]
so the assumption in Theorem~\ref{thm:local} (each $L_i(\Bar X)$ is a singleton)
is satisfied.

Next, we compute the functions $\phi_\ell$ in~\eqref{eq:phi-ell} at $\Bar X$.
Since $n=2$ and $k=2$, for $\ell=1$ we have
\[
\phi_1(x)
=
\sum_{i\in I_1(\Bar X)}\rho(x-a_i)
+
\frac{\lambda n}{2}
\sum_{j=2}^2 \rho^2(x-\bar x_j)
=|x|+\frac12(x-2)^2,\quad x\in\R.
\]
Similarly, for $\ell=2$ we have
\[
\phi_2(x)
=
\sum_{i\in I_2(\Bar X)}\rho(x-a_i)
+
\frac{\lambda n}{2}
\sum_{r=1}^1 \rho^2(\bar X_r-x)
=|x-2|+\frac12 x^2,\quad x\in\R.
\]
Clearly,
$\phi_1$ and $\phi_2$ are convex on $\R$.
 A direct computation shows that
\[
\arg\min \phi_1 = \{1\}\quad\text{and}\quad \arg\min \phi_2 = \{1\},
\]
while the centers at $\Bar X$ are
$\bar x_1=0$ and $\bar x_2=2.$
Thus neither $\bar x_1$ nor $\bar x_2$ is a global minimizer of the corresponding
convex problem~\eqref{eq:phi-ell}.
By Theorem~\ref{thm:local},  $\Bar X=(0,2)$ \emph{cannot} be a local optimal solution of
problem~\eqref{maxminn} when $\lambda=\tfrac12$, even though each
$L_i(\Bar X)$ is a singleton. This illustrates the \emph{necessity} of the
condition in Theorem~\ref{thm:local}.

\emph{Direct verification that $\Bar X$ is not a local minimizer for $\lambda=\tfrac12$.}
Now perturb only the first center and keep the second one fixed.
For $t\in(0,1)$, consider $X(t)=(x_1,x_2)=(t,2).$
For the demand point $a_1=0$, we have
\[|x_1-a_1|=|t-0|=t\quad\text{and}\quad |x_2-a_1|=|2-0|=2,\] and hence 
\[\min_{\ell=1,2}|x_\ell-a_1|=t.\]
For the demand point $a_2=2$, we have
\[|x_1-a_2|=|t-2|=2-t\quad\text{and}\quad |x_2-a_2|=|2-2|=0,\]
and hence 
\[\min_{\ell=1,2}|x_\ell-a_2|=0.\]
The penalty term is
\[\frac12\,|x_1-x_2|^2=\frac12\,|t-2|^2=\frac12\,(2-t)^2.\]
Consequently,
\[f(X(t))=t+\frac12\,(2-t)^2=
2 - t + \frac12\,t^2.\]
We compare this value with $f(\Bar X)=2$:
\[f(X(t)) - f(\Bar X)=
\bigl(2 - t + \tfrac12 t^2\bigr) - 2=
-t\Bigl(1-\frac{t}{2}\Bigr).\]
For $t\in(0,1)$, we have $1-\tfrac{t}{2}>0$, and hence
\[f(X(t)) - f(\Bar X) < 0.\]
In other words,
\[f(X(t)) < f(\Bar X)
\qquad\text{for all }t\in(0,1).\]
Moreover, $X(t)\to\Bar X$ as $t\to 0^+$, so there are points arbitrarily
close to $\Bar X$ with strictly smaller objective value. Therefore,
$\Bar X=(0,2)$ is \emph{not} a local optimal solution of problem~\eqref{maxminn} when $\lambda=\tfrac12$. 
This direct verification is consistent with Theorem~\ref{thm:local}: in this
example $L_i(\Bar X)$ is a singleton for each $i$, but the centers $\bar x_1,\bar x_2$ do not minimize the corresponding convex functions
$\phi_1,\phi_2$, and hence $\Bar X$ cannot be locally optimal.}

\end{example}

We conclude this subsection with an example that illustrates Theorem~\ref{thm:local}
and shows that the set of global optimal solutions is strictly contained in the
set of local optimal solutions.

\begin{example}\label{ex:local-lambda>0}
{\rm We work in $\R$ and take
$\rho(x)=|x|$. Consider the data set $A=\{0,3,4\}$
and $k=2$ centers. For $X=(x_1,x_2)\in\R^{2\times 1}$, the objective function
in~\eqref{maxminn} reads
\[
f(x_1,x_2)
=
\sum_{i=1}^3 \min_{\ell=1,2}|x_\ell-a_i|
+\frac{3\lambda}{2}\,|x_1-x_2|^2.
\]
We show that for small $\lambda>0$, there exists a point which is a
local, but not global, minimizer of $f$, that is,
$$S\subsetneq S_{loc}.$$

\medskip\noindent
\textbf{Step 1:}
Let $\Bar X=(\bar x_1,\bar x_2)=(3,4).$
Then a direct computation shows that
\[I_1(\Bar X)=\{1,2\}\quad\text{and}\quad
I_2(\Bar X)=\{3\}.\]
By the
continuity of the absolute value, there exists $\epsilon>0$ such that, whenever
\[
|x_1-3|<\epsilon
\quad\text{and}\quad
|x_2-4|<\epsilon,
\]
the same cluster assignment is preserved, i.e.,
\[
I_1(X)=\{1,2\}\quad\text{and}\quad I_2(X)=\{3\}.
\]
Moreover, shrinking $\epsilon$ if necessary, we may assume that $x_1>0$ whenever
$|x_1-3|<\epsilon$. This ensures that $|x_1|=x_1$ in this neighbourhood.
For any $X=(x_1,x_2)$ satisfying $|x_1-3|<\epsilon$ and $|x_2-4|<\epsilon$, we have
\[
\min_{\ell=1,2}|x_\ell-a_1|=|x_1-0|=|x_1|,
\quad
\min_{\ell=1,2}|x_\ell-a_2|=|x_1-3|,
\quad
\min_{\ell=1,2}|x_\ell-a_3|=|x_2-4|.
\]
Hence, in this neighbourhood,
\begin{equation}\label{eq:g-local}
f(x_1,x_2)
=
|x_1|+|x_1-3|+|x_2-4|
+\frac{3\lambda}{2}(x_1-x_2)^2.
\end{equation}
Moreover, we obtain
\[
f(\Bar X)
=
|3|+|3-3|+|4-4|
+\frac{3\lambda}{2}(3-4)^2
=
3+\frac{3\lambda}{2}.
\]
To study local optimality, write
\[
x_1=3+\delta_1,
\qquad
x_2=4+\delta_2,
\]
with $|\delta_1|<\epsilon$ and $|\delta_2|<\epsilon$. Since $x_1>0$, in this neighbourhood
we have $|x_1|=3+\delta_1$. Then it follows from~\eqref{eq:g-local} that
\[
\begin{aligned}
f(3+\delta_1,4+\delta_2)
&=
(3+\delta_1) + |\delta_1| + |\delta_2|
+\frac{3\lambda}{2}\bigl[(3+\delta_1)-(4+\delta_2)\bigr]^2\\
&=
3+\delta_1 + |\delta_1| + |\delta_2|
+\frac{3\lambda}{2}\,(-1+\delta_1-\delta_2)^2.
\end{aligned}
\]
Therefore
\[
\begin{aligned}
\Delta(\delta_1,\delta_2)
&= f(3+\delta_1,4+\delta_2) - f(3,4)\\
&=
\delta_1 + |\delta_1| + |\delta_2|
+\frac{3\lambda}{2}\Bigl[(-1+\delta_1-\delta_2)^2-1\Bigr]\\
&=\delta_1 + |\delta_1| + |\delta_2|
-3\lambda(\delta_1-\delta_2)
+\frac{3\lambda}{2}(\delta_1-\delta_2)^2\\
&=(1-3\lambda)\delta_1 + 3\lambda\delta_2
+|\delta_1|+|\delta_2|
+\frac{3\lambda}{2}(\delta_1-\delta_2)^2.
\end{aligned}\]
Note that we always have
\[
(1-3\lambda)\delta_1 + 3\lambda\delta_2
\ge
-(1-3\lambda)|\delta_1| - 3\lambda|\delta_2|.
\]
Substituting into the expression for $\Delta$ and discarding the nonnegative
quadratic term, we obtain
\[
\begin{aligned}
\Delta(\delta_1,\delta_2)
&\ge
-(1-3\lambda)|\delta_1| - 3\lambda|\delta_2|
+|\delta_1|+|\delta_2|\\
&=
3\lambda|\delta_1| + (1-3\lambda)|\delta_2|.
\end{aligned}
\]
If $0<\lambda<\tfrac13$, then $3\lambda>0$ and $1-3\lambda>0$. Therefore,
for any $(\delta_1,\delta_2)\ne(0,0)$ we have
\[
\Delta(\delta_1,\delta_2)>0,
\]
which yields that
\[
f(3+\delta_1,4+\delta_2)
>
f(3,4)
\qquad
\text{for all }(\delta_1,\delta_2)\ne(0,0)
\text{ with }|\delta_1|,|\delta_2|<\epsilon.
\]
In other words, for every $\lambda\in(0,\tfrac13)$, the point $\Bar X=(3,4)$
is a \emph{strict local minimizer} of $f$.

\medskip\noindent
\textbf{Step 2:}
We now compare $\Bar X$ with $X_0=(0,3)$. We first have
\[
f(X_0)=1+13.5\,\lambda.
\]
Therefore,
\[
f(X_0)<f(\Bar X)
\quad\Longleftrightarrow\quad
1+13.5\lambda<3+1.5\lambda
\quad\Longleftrightarrow\quad
\lambda<\frac{1}{6}.
\]

\medskip\noindent
\textbf{Conclusion.}
Combining Step~1 and Step~2, we see that for every $\lambda\in(0,\tfrac16)$,
the point $\Bar X=(3,4)$ satisfies:
\begin{itemize}
\item $\Bar X$ is a strict local minimizer of $f$ (by the direct
computation in Step~1);
\item $\Bar X$ is not a global minimizer of $f$, since
$f(X_0)<f(\Bar X)$ (Step~2).
\end{itemize}
Hence, for any $\lambda\in(0,\tfrac16)$, we have
\[
S\subsetneq S_{loc}.
\]}
\end{example}

\section{Stability of the Optimal Value Function and the Solution Mapping} \label{section3}

In this section, we study the dependence of the optimal value on the data set
\[
A = \{a_1,\dots,a_n\}\subset\R^d.
\]
For convenience, we identify the data set $A$ with the data matrix $\mathbf{A} = \{a_1, \dots, a_n\}\in\R^{n\times d}$. We also equip both $\R^{n\times d}$ and
$\R^{k\times d}$ with the Frobenius norm
\[
\|\mathbf{A}\|_F = \sqrt{\sum_{i=1}^n \|a_i\|^2}
\quad\text{and}\quad
\|X\|_F = \sqrt{\sum_{r=1}^k \|X_r\|^2}.
\]
To emphasize the role of the data matrix $\mathbf{A}$, we write the objective function of \eqref{maxminn} as
\[
f(X;\mathbf{A})  
=
D(X;\mathbf{A})+C(X)\quad\text{for }X=(x_1,\dots,x_k)\in\R^{k\times d}
\]
where
\[
D(X;\mathbf{A})
=
\sum_{i=1}^n \min_{1\le \ell \le k}\rho(x_\ell-a_i)
\quad\text{and}\quad
C(X)
=
\frac{\lambda n}{2}\sum_{1\le s<t\le k}\|x_s-x_t\|^2.
\]
Note that the penalty term $C(X)$ does not depend on $\mathbf{A}$. 

\begin{definition}
The \textit{optimal value function} $\mathcal{V} \colon \mathbb{R}^{n \times d} \to \mathbb{R}$ is defined by
\begin{equation}\label{valueoptimalfunction}
\mathcal{V}(\mathbf{A}) = \inf_{X \in \mathbb{R}^{k \times d}} f(X; \mathbf{A}).
\end{equation}
Furthermore, the \textit{global optimal solution mapping}  $S \colon \mathbb{R}^{n \times d} \rightrightarrows \mathbb{R}^{k \times d}$, which assigns each data matrix $\mathbf{A}\in \mathbb{R}^{n \times d}$ to the set of its global minimizers for problem \eqref{maxminn}, is given by
\begin{equation}\label{solutionmapping}
S(\mathbf{A}) = \left\{ \Bar{X} \in \mathbb{R}^{k \times d} \mid f(\Bar{X}; \mathbf{A}) = \mathcal{V}(\mathbf{A}) \right\}.
\end{equation}
\end{definition}

We first establish a Lipschitz estimate in the data for the data-fitting part
$D(X;\mathbf{A})$ with $X$ fixed.

\begin{lemma}\label{lem:Lip-D}
Let $X\in\R^{k\times d}$ be fixed. Then the mapping
\[
\R^{n\times d}\ni \mathbf{A}\ \longmapsto\ D(X;\mathbf{A})
\]
is globally Lipschitz continuous with
\[
|D(X;\mathbf{A})-D(X;\mathbf{B})|
\;\le\;
\|F^\circ\|\sqrt{n}\,\|\mathbf{A}-\mathbf{B}\|_F
\qquad
\text{for all }\mathbf{A},\mathbf{B}\in\R^{n\times d}.
\]
In particular, the Lipschitz constant does not depend on $X$.
\end{lemma}

\begin{proof}
Consider two arbitrary data matrices
$\mathbf{A}=(a_1,\dots,a_n)$, $\mathbf{B}=(b_1,\dots,b_n)\in\R^{n\times d}$. For
each $i\in\{1,\dots,n\}$, we define
\[
\psi_i(\mathbf{A})
=
\min_{1\le \ell \le k}\rho(x_\ell-a_i)
\quad\text{and}\quad
\psi_i(\mathbf{B})
=
\min_{1\le \ell \le k}\rho(x_\ell-b_i).
\]
Then
\[
D(X;\mathbf{A})
=
\sum_{i=1}^n \psi_i(\mathbf{A})
\quad\text{and}\quad
D(X;\mathbf{B})
=
\sum_{i=1}^n \psi_i(\mathbf{B}),
\]
which implies that
\[
\bigl|D(X;\mathbf{A})-D(X;\mathbf{B})\bigr|
=
\Bigl|\sum_{i=1}^n\psi_i(\mathbf{A})-\sum_{i=1}^n\psi_i(\mathbf{B})\Bigr|
\le
\sum_{i=1}^n|\psi_i(\mathbf{A})-\psi_i(\mathbf{B})|.
\]
Fix $i\in\{1,\dots,n\}$. Since $\psi_i$ is the minimum of finitely many
$\|F^\circ\|$--Lipschitz functions in $a_i$, we have
\[
|\psi_i(\mathbf{A})-\psi_i(\mathbf{B})|
\le
\max_{1\le \ell \le k}
\bigl|
\rho(x_\ell-a_i)-\rho(x_\ell-b_i)
\bigr|.
\]
Using Lemma \ref{lemmarho0}(b) and (g), we obtain
\[
\bigl| \rho(x_\ell-a_i)-\rho(x_\ell-b_i) \bigr| \le \rho(a_i-b_i) \le \|F^\circ\| \|a_i-b_i\|
\]
for all $\ell=1,\dots,k$. Thus, it follows that
\[
|\psi_i(\mathbf{A})-\psi_i(\mathbf{B})| \le \|F^\circ\| \|a_i-b_i\|.
\]
Summing over $i=1,\dots,n$ and applying the Cauchy–Schwarz inequality, we have
\[
|D(X;\mathbf{A})-D(X;\mathbf{B})| \le \|F^\circ\| \sum_{i=1}^n \|a_i-b_i\| \le \|F^\circ\| \sqrt{n} \|\mathbf{A}-\mathbf{B}\|_F.
\]
This completes
the proof.
\end{proof}

We now establish a global Lipschitz property for the optimal value function
$\mathcal{V}$ defined in \eqref{valueoptimalfunction}.

\begin{theorem}\label{thm:Lip-vk}
The optimal value function $\mathcal{V}(\cdot)$ of problem \eqref{maxminn}
is globally Lipschitz continuous with respect to the Frobenius norm. More
precisely,
\[
|\mathcal{V}(\mathbf{A})-\mathcal{V}(\mathbf{B})|
\;\le\;
\|F^\circ\|\sqrt{n}\,\|\mathbf{A}-\mathbf{B}\|_F
\qquad
\text{for all }\mathbf{A},\mathbf{B}\in\R^{n\times d}.
\]
\end{theorem}

\begin{proof}
Let $\mathbf{A},\mathbf{B}\subset\R^{n\times d}$ and fix $\varepsilon>0$. By the definition of $\mathcal{V}(\mathbf{A})$,
there exists $X_\varepsilon\in\R^{k\times d}$ such that
\[
f(X_\varepsilon;\mathbf{A})\le \mathcal{V}(\mathbf{A})+\varepsilon.
\]
Then we have
\[
\mathcal{V}(\mathbf{B})
\le
f(X_\varepsilon;\mathbf{B})=D(X_\varepsilon;\mathbf{B})+C(X_\varepsilon)\]
and
\[\mathcal{V}(\mathbf{A})\le f(X_\varepsilon;\mathbf{A})=D(X_\varepsilon;\mathbf{A})+C(X_\varepsilon)\le \mathcal{V}(\mathbf{A})+\varepsilon.\]
Subtracting, we obtain
\[\mathcal{V}(\mathbf{B})-\mathcal{V}(\mathbf{A})\le D(X_\varepsilon;\mathbf{B})-D(X_\varepsilon;\mathbf{A})+\varepsilon.\]
Meanwhile,  it follows from Lemma~\ref{lem:Lip-D} that
\[D(X_\varepsilon;\mathbf{B})-D(X_\varepsilon;\mathbf{A})\le\|F^\circ\|\sqrt{n}\,\|\mathbf{B}-\mathbf{A}\|_F,
\]
and so
\[
\mathcal{V}(\mathbf{B})-\mathcal{V}(\mathbf{A})\le
\|F^\circ\|\sqrt{n}\,\|\mathbf{B}-\mathbf{A}\|_F+\varepsilon.
\]
Since $\varepsilon>0$ is arbitrary, this yields
\[\mathcal{V}(\mathbf{B})-\mathcal{V}(\mathbf{A})\le
\|F^\circ\|\sqrt{n}\,\|\mathbf{B}-\mathbf{A}\|_F.
\]
Interchanging the roles of $\mathbf{A}$ and $\mathbf{B}$, we also obtain
\[
\mathcal{V}(\mathbf{A})-\mathcal{V}(\mathbf{B})
\le\|F^\circ\|\sqrt{n}\,\|\mathbf{A}-\mathbf{B}\|_F,\]
and hence
\[|\mathcal{V}(\mathbf{A})-\mathcal{V}(\mathbf{B})|\le
\|F^\circ\|\sqrt{n}\,\|\mathbf{A}-\mathbf{B}\|_F.\]
This proves the theorem.
\end{proof}

Our next goal is to investigate the semicontinuity properties of the global optimal solution mapping \eqref{solutionmapping}. 
 We begin by recalling the definition of upper semicontinuity for set-valued mappings, see \cite[p. 109]{Berge}.

\begin{definition}\label{def:usc-neighborhood}
Let $G: \mathbb{R}^{n \times d} \rightrightarrows \mathbb{R}^{k \times d}$ be a set-valued mapping. We say that $G$ is \emph{upper semicontinuous} at $\mathbf{\Bar{A}} \in \mathbb{R}^{n \times d}$ if for every open set $U \subset \mathbb{R}^{k \times d}$ such that $G(\mathbf{\Bar{A}}) \subset U$, there exists an open neighborhood $V$ of $\mathbf{\Bar{A}}$ such that 
\[
G(\mathbf{A}) \subset U \quad \text{for all } \mathbf{A} \in V.
\]
The mapping $G$ is said to be upper semicontinuous on $\mathbb{R}^{n \times d}$ if it is upper semicontinuous at every point $\mathbf{A} \in \mathbb{R}^{n \times d}$.
\end{definition}

\begin{theorem}
The global optimal solution mapping $S: \mathbb{R}^{n \times d} \rightrightarrows \mathbb{R}^{k \times d}$ defined in \eqref{solutionmapping} is upper semicontinuous.
\end{theorem}

\begin{proof}
Fix any $\Bar{\mathbf{A}} \in \mathbb{R}^{n \times d}$. Suppose on the contrary that $S$ is not upper semicontinuous at $\Bar{\mathbf{A}}$. Then there exists an open set $U \subset \mathbb{R}^{k \times d}$ with $S(\Bar{\mathbf{A}}) \subset U$ such that for every $m \in \mathbb{N}$, one can find $\mathbf{A}_m \in B(\Bar{\mathbf{A}}, 1/m)$ and $X_m \in S(\mathbf{A}_m)$ with
\[
X_m \notin U.
\]
Since $\mathbf{A}_m \to \Bar{\mathbf{A}}$ and the optimal value function $\mathcal{V}$ is continuous by Theorem~\ref{thm:Lip-vk}, we have
\[
f(X_m; \mathbf{A}_m) = \mathcal{V}(\mathbf{A}_m) \longrightarrow \mathcal{V}(\bar{\mathbf{A}}),
\]
so the sequence $\{f(X_m; \mathbf{A}_m)\}_m$ is bounded.
We now show that the sequence $\{X_m\}_m$ is bounded. Assume on the contrary that it is unbounded. Then there exists a subsequence (not relabelled) such that $$\|X_m\|_F \to \infty.$$ By Lemma~\ref{lem:Lip-D}, there is a constant $L > 0$ such that
\[
|f(X_m; \mathbf{A}_m) - f(X_m; \bar{\mathbf{A}})| \le L \|\mathbf{A}_m - \Bar{\mathbf{A}}\|_F \longrightarrow 0,
\]
which implies that $\{f(X_m; \Bar{\mathbf{A}})\}_m$ is also bounded. On the other hand, for the fixed data matrix $\Bar{\mathbf{A}}$, the proof in Theorem~\ref{thm1} shows that $f(\cdot; \Bar{\mathbf{A}})$ is coercive in $X$, so $\|X_m\|_F \to \infty$ forces $f(X_m; \Bar{\mathbf{A}}) \to \infty$, which is a contradiction. Therefore, $\{X_m\}_m$ must be bounded. By the Bolzano--Weierstrass theorem, there exists a subsequence (still denoted by $\{X_m\}$) and some $\Bar{X} \in \mathbb{R}^{k \times d}$ such that $X_m \to \Bar{X}$. Since each $X_m \notin U$ and the complement of $U$ is closed, we have
\begin{equation}\label{eq:usc-contradict}
\bar{X} \notin U.
\end{equation}
Because $f$ is continuous in both variables, we obtain
\[
f(X_m; \mathbf{A}_m) \longrightarrow f(\Bar{X}; \Bar{\mathbf{A}}).
\]
Recalling that $f(X_m; \mathbf{A}_m) = \mathcal{V}(\mathbf{A}_m)$ and $\mathcal{V}(\mathbf{A}_m) \to \mathcal{V}(\Bar{\mathbf{A}})$, we conclude that
\[
f(\Bar{X}; \Bar{\mathbf{A}}) = \mathcal{V}(\Bar{\mathbf{A}}),
\]
which means $\Bar{X} \in S(\Bar{\mathbf{A}}) \subset U$. This contradicts \eqref{eq:usc-contradict}. Hence, $S$ is upper semicontinuous at $\Bar{\mathbf{A}}$. Since $\Bar{\mathbf{A}}$ was arbitrary, the mapping $S$ is upper semicontinuous on $\mathbb{R}^{n \times d}$.
\end{proof}

\section{Computational Framework via Smoothing and DC Programming}\label{section4}

In this section, we develop a DC-based algorithmic framework for minimizing the nonsmooth, nonconvex objective function $f$. The approach is built on an exact DC decomposition, Nesterov-type smoothing, and efficient computation of gradients and subgradients of the DC components. 
Then we use this model to determine the optimal number of centers in multifacility location problems. 

\subsection{DC Decomposition and Nesterov Smoothing}\label{subsec:DC-smoothing}
We begin by expressing the objective function of \eqref{maxminn} in the form of a
difference-of-convex (DC) decomposition. 
Observe first that for each $i=1,\ldots,n$, the assignment term can be written as
\begin{equation}\label{DC_identity}
\min_{1\le \ell \le k} \rho(x_\ell - a_i)
=
\sum_{\ell=1}^k \rho(x_\ell - a_i)
-
\max_{r=1,\ldots,k}
\sum_{\substack{\ell=1,\;\ell\neq r}}^k \rho(x_\ell - a_i).
\end{equation}
Substituting \eqref{DC_identity} into \eqref{maxminn} yields
\begin{align*}
f(X)
&=
\sum_{i=1}^n \left[
\sum_{\ell=1}^k \rho(x_\ell - a_i)
-
\max_{r=1,\ldots,k}
\sum_{\substack{\ell=1,\;\ell\neq r}}^k \rho(x_\ell - a_i)
\right]
+
\frac{\lambda n}{2} \sum_{1\le s<t\le k} \|x_s-x_t\|^2 \\[0.1cm]
&=
\Bigg(
\sum_{i=1}^n \sum_{\ell=1}^k \rho(x_\ell - a_i)
+
\frac{\lambda n}{2} \sum_{1\le s<t\le k} \|x_s-x_t\|^2
\Bigg)
-
\Bigg(
\sum_{i=1}^n
\max_{r=1,\ldots,k}
\sum_{\substack{\ell=1,\;\ell\neq r}}^k \rho(x_\ell - a_i)
\Bigg),
\end{align*}
that is,
\[
f(X) = g(X) - h(X),
\]
with
\begin{align*}
g(X)
&=
\sum_{i=1}^n \sum_{\ell=1}^k \rho(x_\ell - a_i)
+
\frac{\lambda n}{2}
\sum_{1\le s<t\le k}\|x_s-x_t\|^2,
\\[0.05cm]
h(X)
&=
\sum_{i=1}^n
\max_{r=1,\ldots,k}
\sum_{\substack{\ell=1,\;\ell\neq r}}^k \rho(x_\ell - a_i).
\end{align*}
Clearly, both $g$ and $h$ are convex, so $f$ is a DC function.

\begin{remark}
Direct DC algorithms based on convex conjugates would require an explicit expression for the Fenchel conjugate of $\rho$, which is generally unavailable. To obtain a practical scheme, we instead smooth the gauge $\rho$ and work with a differentiable approximation while preserving the DC structure.
\end{remark}

Recall from Lemma \ref{lemmarho0}(f) that
\[
\rho(z) = \max_{y \in F^\circ} \langle z, y \rangle.
\]
Following the Nesterov smoothing technique introduced in \cite{nesterov2005smooth} and developed in \cite{Nam2017,Nam2018}, for a parameter $\mu > 0$, a smooth approximation of the Minkowski gauge function is given by
\begin{equation}\label{eq:rho-smooth-general}
\rho_\mu(z)=\max_{y \in F^\circ} \left\{
\langle z, y \rangle - \frac{\mu}{2} \|y\|^2
\right\}=\frac{1}{2\mu} \|z\|^2-\frac{\mu}{2} \left[ d\!\left( \frac{z}{\mu}; F^\circ \right) \right]^2.
\end{equation}
Next, we define a smoothed objective function by replacing the gauge function $\rho$ with its smooth approximation $\rho_\mu$ as follows:
\begin{equation}
    \label{eq:fkmu}
f_\mu(X)
=
\sum_{i=1}^n \sum_{\ell=1}^k
\rho_\mu(x_\ell - a_i)
+
\frac{\lambda n}{2}
\sum_{1\le s<t\le k}\|x_s-x_t\|^2-\sum_{i=1}^n
\max_{r=1,\ldots,k}
\sum_{\substack{\ell=1,\;\ell\neq r}}^k \rho_\mu(x_\ell - a_i).
\end{equation} 
Using \eqref{eq:rho-smooth-general}, we obtain a DC decomposition $f_\mu = g_\mu - h_\mu$ with
\begin{equation}\label{gmu}
g_\mu(X)
=
\frac{1}{2\mu}
\sum_{i=1}^n \sum_{\ell=1}^k
\|x_\ell - a_i\|^2
+
\frac{\lambda n}{2}
\sum_{1\le s<t\le k}\|x_s-x_t\|^2    
\end{equation}
and
\begin{equation}\label{hmu}
 h_\mu(X)
=
\frac{\mu}{2}
\sum_{i=1}^n \sum_{\ell=1}^k
\left[ d\!\left( \frac{x_\ell - a_i}{\mu}; F^\circ \right) \right]^2
+
\sum_{i=1}^n
\max_{r=1,\dots,k}
\sum_{\substack{\ell=1,\; \ell \neq r}}^k \rho_\mu(x_\ell - a_i).   
\end{equation}
Here, both $g_\mu$ and $h_\mu$ are convex functions. Moreover, since $\rho_\mu(z) \to \rho(z)$ as $\mu \downarrow 0$, it follows that $f_\mu(X) \to f(X)$ pointwise for any $X \in \mathbb{R}^{k \times d}$.

\subsection{Subgradient of the Concave Part and DCA Step} \label{compY}

A crucial step in each iteration of the DCA and BDCA schemes is the computation of a subgradient
$Y(X) \in \partial h_{\mu}(X)$. For computational convenience, we decompose
$h_\mu$ as $h_\mu = h^{(1)}_{\mu} + h^{(2)}$, where
\[
h^{(1)}_{\mu}(X)
=
\frac{\mu}{2} \sum_{i=1}^n \sum_{\ell=1}^k
\left[
d\!\left(\frac{x_\ell-a_i}{\mu};F^\circ\right)
\right]^2
\quad\text{and}\quad
h^{(2)}(X)
=
\sum_{i=1}^n
\max_{r=1,\dots,k}
\sum_{\substack{\ell=1\\ \ell\neq r}}^k
\rho(x_\ell-a_i).
\]

\paragraph{Gradient of $h^{(1)}_{\mu}$.}
The function $h^{(1)}_{\mu}$ is continuously differentiable
(see, e.g., \cite[Proposition~3.1]{Nam2017}), and its gradient with respect to
$x_p$ is given by
\[
\nabla_{x_p} h^{(1)}_{\mu}(X)
=
\sum_{i=1}^n
\left(
\frac{x_p-a_i}{\mu}
-
P\!\left(\frac{x_p-a_i}{\mu};F^\circ\right)
\right).
\]

\paragraph{Subgradients of $h^{(2)}$.}
For each $i=1,\ldots,n$, define
\[
Q_{i,r}(X)
=
\sum_{\ell=1,\ \ell\neq r}^{k}
\rho(x_\ell-a_i),
\quad\text{and}\quad
Q_i(X)
=
\max_{r=1,\ldots,k} Q_{i,r}(X).
\]
By construction,
\[
h^{(2)}(X)=\sum_{i=1}^n Q_i(X).
\]
Since $h^{(2)}$ is a finite sum of convex functions, its subdifferential satisfies
\[
\partial h^{(2)}(X)=\sum_{i=1}^n \partial Q_i(X).
\]
Therefore, it suffices to compute a subgradient
$V_i\in\partial Q_i(X)$ for each $i=1,\ldots,n$.

Fix $i\in\{1,\ldots,n\}$. To compute $V_i$, we proceed as follows:
\begin{itemize}
\item[$\circ$]
Choose any active index
\[
r_i \in \arg\max_{r=1,\ldots,k} Q_{i,r}(X).
\]
\item[$\circ$]
Define
\[
V_i=(v_{i1},\ldots,v_{ik})
\]
with components given by
\begin{itemize}
\item $v_{i r_i}=0$, since $x_{r_i}$ does not appear in $Q_{i,r_i}(X)$;
\item for each $\ell\neq r_i$,
\[
v_{i\ell} \in \partial\rho (x_\ell-a_i).
\]
\end{itemize}
\end{itemize}
Then $V_i\in\partial Q_i(X)$, and consequently
\[
\sum_{i=1}^n V_i \in \partial h^{(2)}(X).
\]
A subgradient $Y(X) \in \partial h_\mu(X)$ is obtained by summing the subgradients of these two components: $$Y = Y^1 + Y^2,$$ where $Y^1 \in \partial h^{1}_\mu(X)$ and $Y^2 \in \partial h^{(2)}(X)$.

\subsection{DCA Iterations and Solution of the Subproblem}

For a fixed smoothing parameter $\mu > 0$, we consider the smoothed DC decomposition $$f_\mu(X) = g_\mu(X) - h_\mu(X).$$ Given the current iterate $X^{(t)}$, following the standard DCA framework \cite{LeThi2018,LeThi2024,LeThi1996}, the update is defined by
\begin{equation}\label{eq:dca-step}
Y^{(t)} \in \partial h_\mu(X^{(t)}), \qquad
X^{(t+1)} \in \arg\min_{X \in \mathbb{R}^{k \times d}} \{ g_\mu(X) - \langle Y^{(t)}, X \rangle \}.
\end{equation}
Since $g_\mu$ is strongly convex in our formulation, the minimizer $X^{(t+1)}$ is unique and can be characterized by the first-order optimality condition 
\begin{equation}\label{condition}
\nabla g_\mu(X^{(t+1)}) = Y^{(t)}.    
\end{equation}

\paragraph{The gradient of  $g_\mu$.}
In our specific case with the squared Euclidean fusion penalty, the convex component is given by
\[
g_\mu(X) = \frac{1}{2\mu} \sum_{i=1}^n \sum_{\ell=1}^k \|x_\ell - a_i\|^2
+ \frac{\lambda n}{2} \sum_{1 \le s < t \le k} \|x_s - x_t\|^2.
\]
Differentiating $g_\mu$ with respect to each prototype $x_p$ yields
\begin{equation}\label{eq:grad_gmu_final}
\nabla_{x_p} g_\mu(X)
=
\frac{n}{\mu}(x_p - \bar{\mathbf{A}})
+
\lambda n \sum_{j \neq p} (x_p - x_j),
\end{equation}
where $\bar{\mathbf{A}} = \frac{1}{n} \sum_{i=1}^n a_i$ is the centroid of the data points.

\paragraph{DCA update via linear systems.}
Substituting \eqref{eq:grad_gmu_final} into the optimality condition \eqref{condition}, the DCA subproblem reduces to solving a coupled linear system for $p=1,\dots,k$:
\begin{equation}\label{eq:linear-system}
\left( \frac{1}{\mu} + \lambda(k-1) \right) n\, x_p^{(t+1)}
-
\lambda n \sum_{j \neq p} x_j^{(t+1)}
=
y_p^{(t)} + \frac{n}{\mu} \bar{\mathbf{A}}.
\end{equation}
Set $$B_p = y_p^{(t)} + \frac{n}{\mu} \bar{\mathbf{A}}\quad\text{and}\quad\sigma^{(t+1)} = \sum_{j=1}^k x_j^{(t+1)}.$$ 
Summing \eqref{eq:linear-system} over all $p$ yields $\sigma^{(t+1)} = \frac{\mu}{n} \sum_{p=1}^k B_p$. Consequently, each prototype can be updated by
\begin{equation} \label{eq:closed-form-update}
x_p^{(t+1)} = \frac{B_p + \lambda n \sigma^{(t+1)}}{n(\mu^{-1} + \lambda k)}, \quad p=1,\dots,k.
\end{equation}
This formulation avoids the need for matrix inversion or iterative solvers within the DCA loop, reducing the update cost to $\mathcal{O}(kd)$.

\section{Computational Framework via DCA-Type Algorithms} \label{section5}

We now summarize the resulting DCA-type schemes for minimizing the smoothed objective $f_\mu = g_\mu - h_\mu$. At each iteration, a subgradient $Y^{(t)} \in \partial h_\mu(X^{(t)})$ is computed via Subsection~\ref{compY}. The next iterate $X^{(t+1)}$ is then obtained by solving the linear system $\nabla g_\mu(X^{(t+1)}) = Y^{(t)}$, which admits a unique solution due to the strong convexity of $g_\mu$.

\subsection{Basic DCA Scheme}

The basic DCA leverages the linear structure of $\nabla g_\mu$ to perform fast updates.

\begin{algorithm}[htbp]
\caption{DCA for the smoothed Minkowski clustering model}
\label{alg:DCA}
\begin{algorithmic}[1]
\Require Data $\mathbf{A}=(a_1,\dots,a_n)$; polar set $F^\circ$; parameters $\lambda>0$, $\mu>0$; tolerance $\varepsilon>0$.
\Ensure Prototype matrix $X^\star \in \mathbb{R}^{k \times d}$.
\State Choose an initial $X^{(0)} \in \mathbb{R}^{k \times d}$.
\State Precompute $\bar{\mathbf{A}} = \frac{1}{n} \sum_{i=1}^n a_i$.
\For{$t = 0,1,2,\dots$}
    \State Compute $Y^{(t)} \in \partial h_\mu(X^{(t)})$ via Subsection~\ref{compY}.
    \State Set $B_p = y_p^{(t)} + \frac{n}{\mu} \bar{\mathbf{A}}$ for $p=1,\dots,k$.
    \State Compute total sum $\sigma^{(t+1)} = \frac{\mu}{n} \sum_{p=1}^k B_p$.
    \State Update each prototype: 
    \[ x_p^{(t+1)} = \frac{B_p + \lambda n \sigma^{(t+1)}}{n(\mu^{-1} + \lambda k)}, \quad p=1,\dots,k. \]
    \If{$\|X^{(t+1)} - X^{(t)}\|_F < \varepsilon$}
        \State \Return $X^\star = X^{(t+1)}$
    \EndIf
\EndFor
\end{algorithmic}
\end{algorithm}

\subsection{Boosted and Inertial Variants}

To exploit the smooth structure of $g_\mu$ more effectively and accelerate convergence, we consider two standard variants: Boosted DCA (BDCA) and Multi-step Inertial DCA (M-IDCA).

\paragraph{Boosted DCA.}
BDCA incorporates a line search along the direction $d_t = Z^{(t)} - X^{(t)}$, where $Z^{(t)}$ is the standard DCA point. This ensures a larger decrease in the objective function $f_\mu$ per iteration.

\begin{algorithm}[H]
\caption{BDCA for the smoothed Minkowski clustering model}
\label{alg:BDCA}
\begin{algorithmic}[1]
\Require Data $\mathbf{A}$; polar set $F^\circ$; parameters $\lambda>0$, $\mu>0$; line-search parameters $\alpha>0, \beta \in (0,1)$; tolerance $\varepsilon>0$.
\Ensure Prototype matrix $X^\star$.
\State Initialize $X^{(0)} \in \mathbb{R}^{k \times d}$.
\For{$t = 0,1,2,\dots$}
    \State Compute $Y^{(t)} \in \partial h_\mu(X^{(t)})$ as in Subsection~\ref{compY}.
    \State Compute $Z^{(t)}$ by solving the linear system $\nabla g_\mu(Z^{(t)}) = Y^{(t)}$.
    \State Set $d_t := Z^{(t)} - X^{(t)}$.
    \If{$\|d_t\|_F < \varepsilon$} \State \Return $X^\star = X^{(t)}$ \EndIf
    \State Set $\gamma_t \gets 1$. \Comment{Backtracking line search}
    \While{$f_\mu(X^{(t)} + \gamma_t d_t) > f_\mu(X^{(t)}) - \alpha \gamma_t^2 \|d_t\|_F^2$}
        \State $\gamma_t \gets \beta \gamma_t$
    \EndWhile
    \State $X^{(t+1)} \gets X^{(t)} + \gamma_t d_t$.
    \If{$\|X^{(t+1)} - X^{(t)}\|_F < \varepsilon$} \State \Return $X^\star = X^{(t+1)}$ \EndIf
\EndFor
\end{algorithmic}
\end{algorithm}

\paragraph{Multi-step inertial DCA (M-IDCA).}
The inertial variant incorporates momentum from the last $m$ iterates to speed up the exploration of the search space, while the line search maintains the monotonic descent property.

\begin{algorithm}[H]
\caption{Multi-step inertial DCA (M-IDCA)}
\label{alg:BM-M-IDCA}
\begin{algorithmic}[1]
\Require Data $\mathbf{A}$; polar set $F^\circ$; parameters $\lambda, \mu, m$; inertial coefficients $\alpha_i$; line-search parameters $\alpha, \beta$; tolerance $\varepsilon$.
\Ensure Prototype matrix $X^\star$.
\State Initialize $X^{(0)} \in \mathbb{R}^{k \times d}$.
\For{$t = 0,1,2,\dots$}
    \State Compute $Y^{(t)} \in \partial h_\mu(X^{(t)})$ as in Subsection~\ref{compY}.
    \State Compute $W^{(t)}$ by solving $\nabla g_\mu(W^{(t)}) = Y^{(t)}$. \Comment{DCA point}
    \State Set $Z^{(t)} = W^{(t)} + \sum_{i=1}^{\min(t,m)} \alpha_i ( X^{(t+1-i)} - X^{(t-i)} )$. \Comment{Extrapolation}
    \State Set $d_t := Z^{(t)} - X^{(t)}$.
    \State (Perform line search as in BDCA to update $X^{(t+1)}$).
\EndFor
\end{algorithmic}
\end{algorithm}

\subsection{Automatic Cluster Selection: LDCA-K}

A distinctive feature of our Laplacian model is its ability to eliminate redundant clusters. We combine the DCA updates with a prototype deletion mechanism. Prototypes that do not capture any data points are removed, effectively estimating the optimal number of clusters $k$ automatically.

\begin{algorithm}[H]
\caption{Laplacian DCA with prototype deletion (LDCA-K)}
\label{alg:LDCA-K}
\begin{algorithmic}[1]
\Require Data $\mathbf{A}$; initial $k_0$; parameters $\lambda, \mu$.
\State Initialize $k \gets k_0$ and $X \in \mathbb{R}^{k \times d}$ via $k$-means++.
\Repeat
    \Repeat
        \State Compute $Y \in \partial h_\mu(X)$ as in Subsection~\ref{compY}.
        \State Update $X$ by solving $\nabla g_\mu(X) = Y$ (via DCA, BDCA or M-IDCA).
    \Until{inner convergence}
    \State Assign each $a_i$ to its nearest prototype: $\ell_i \in \arg\min_\ell \rho(x_\ell - a_i)$.
    \State Delete prototypes $x_\ell$ for which $\{i \mid \ell_i = \ell\} = \emptyset$ and update $k$.
\Until{no further deletions occur or max iterations reached}
\end{algorithmic}
\end{algorithm}
\section{Convergence Analysis} \label{section6}

In this section, we establish the fundamental convergence properties of the DCA-type schemes for the smoothed Minkowski clustering model. Throughout the analysis, we fix the parameters $\lambda > 0$ and $\mu > 0$, and consider the DC decomposition $f_\mu = g_\mu - h_\mu$ derived in Section~\ref{section4}.

\begin{definition}
    Let $\Omega \subseteq \mathbb{R}^d$ be a convex set. A function $g: \Omega \to \mathbb{R}$ is said to be strongly convex with parameter $\gamma > 0$ if for all $x, y \in \Omega$ and any $\alpha \in [0, 1]$, the following inequality holds:
\begin{equation*}
g(\alpha x + (1-\alpha)y) \le \alpha g(x) + (1-\alpha)g(y) - \frac{\gamma}{2}\alpha(1-\alpha)|x-y|^2.
\end{equation*}
\end{definition}

\begin{remark}\label{stronglyconvex}
As stated in \cite[p. 459]{boyd2004convex} or in \cite[Definition 2.1.2 and Theorem 2.1.11]{nesterov2018lectures}, we have
\begin{enumerate}
    \item[{\rm(a)}] a continuously differentiable function $g$ is strongly convex with parameter $\gamma > 0$ if it satisfies the following quadratic lower bound:
\begin{equation*}
g(y) \ge g(x) + \langle \nabla g(x), y - x \rangle + \frac{\gamma}{2} \|y - x\|^2\quad \text{for all }x, y \in \Omega
\end{equation*}
 \item[{\rm(b)}] if $g$ is twice continuously differentiable, then $g$
is strongly convex with parameter $\gamma > 0$ 
 if and only if its Hessian matrix $\nabla^2 g(x)$ satisfies
\begin{equation*}
\nabla^2 g(x) \succeq \gamma I_d, \quad \forall x \in \Omega,
\end{equation*}
where $I_d$ is the identity matrix and $\succeq$ denotes the Loewner partial order (i.e., $\nabla^2 g(x) - \gamma I_d$ is positive semidefinite). Since the Hessian matrix is symmetric, this condition is equivalent to requiring that all eigenvalues of  $\nabla^2 g(x)$ are bounded below by $\gamma$.
 \end{enumerate}
\end{remark}

In the following theorem, we establish the fundamental analytical properties of the function $g_\mu$. 

\begin{theorem} \label{thm:smoothness-special}
The function $g_\mu\colon \R^{k\times d}\to \R$ defined by \eqref{gmu} is twice continuously differentiable on $\mathbb{R}^{k\times d}$ and possesses the following properties:
\begin{enumerate}
    \item[{\rm(a)}] The Hessian matrix of $g_\mu$ is given by
    \begin{equation}\label{eq:full-hessian}
    \nabla^2 g_\mu(X) = \frac{n}{\mu} I_{kd} + \lambda n (L_G \otimes I_d), \quad \forall X \in \mathbb{R}^{k \times d}.
    \end{equation}
    
    \item[{\rm(b)}] $g_\mu$ is strongly convex with parameter $\gamma = \frac{n}{\mu}$. That is, for all $X, X' \in \mathbb{R}^{k\times d}$, one has
    \begin{equation*}
    g_\mu(X') \ge g_\mu(X) + \langle \nabla g_\mu(X), X'-X \rangle + \frac{n}{2\mu}\|X'-X\|_F^2.
    \end{equation*}
    
    \item[{\rm(c)}] The gradient $\nabla g_\mu$ is Lipschitz continuous with constant $L = \frac{n}{\mu} + \lambda nk$. That is, for all $X, X' \in \mathbb{R}^{k \times d}$, one has
    \begin{equation*}
    \|\nabla g_\mu(X') - \nabla g_\mu(X)\|_F \le \left( \frac{n}{\mu} + \lambda n k \right) \|X' - X\|_F.
    \end{equation*}
\end{enumerate}
\end{theorem}

\begin{proof}
It follows from \eqref{gmu} that $g_\mu$ is decomposed as
\[
g_\mu(X) = g_\mu^{(1)}(X) + g_\mu^{(2)}(X),
\]
where
\begin{equation*}
g_\mu^{(1)}(X) = \frac{1}{2\mu} \sum_{\ell=1}^k \sum_{i=1}^n \|x_\ell - a_i\|^2
\quad\text{and}\quad 
g_\mu^{(2)}(X) = \frac{\lambda n}{2} \sum_{1 \le s < t \le k} \|x_s - x_t\|^2.
\end{equation*}
Since both $g_\mu^{(1)}$ and $g_\mu^{(2)}$ are quadratic forms in $X$,  $g_\mu$ is twice continuously differentiable on $\mathbb{R}^{k \times d}$.

\noindent(a) We prove this statement in three steps.

\medskip
\noindent\textbf{Step 1: The Hessian matrix of $g_\mu^{(1)}$.}

Write the prototype matrix as
\[
X =
\begin{bmatrix}
x_1^\top  \\
 x_2^\top  \\
 \vdots  \\
 x_k^\top 
\end{bmatrix}
\in \R^{k\times d},
\]
so that $X$ can be viewed as a vector in $\R^{kd}$. Moreover, we have
\[
g_\mu^{(1)}(X)
= \frac{1}{2\mu} \sum_{\ell=1}^k \sum_{i=1}^n \|x_\ell - a_i\|^2
= \frac{1}{2\mu} \sum_{\ell=1}^k \sum_{i=1}^n \big(\|x_\ell\|^2 - 2\langle x_\ell, a_i\rangle + \|a_i\|^2\big).
\]
Note that
\[
\frac{1}{2\mu}\sum_{\ell=1}^k \sum_{i=1}^n \|x_\ell\|^2 = \frac{n}{2\mu}\sum_{\ell=1}^k \|x_\ell\|^2.
\]
Thus, for each $\ell = 1,\dots,k$, the gradient and Hessian with respect to $x_\ell$ are
\[
\nabla_{x_\ell} g_\mu^{(1)}(X) = \frac{n}{\mu} x_\ell - \frac{1}{\mu} \sum_{i=1}^n a_i
\quad\text{and}\quad
\nabla_{x_\ell}^2 g_\mu^{(1)}(X) = \frac{n}{\mu} I_d,
\]
where $I_d$ is the identity matrix. 
Moreover, 
\[
\nabla_{x_\ell x_j}^2 g_\mu^{(1)}(X) = 0 \quad \text{for } \ell \neq j.
\]
Therefore, the full Hessian of $g_\mu^{(1)}$ in block-matrix form is
\begin{equation}\label{hessian1}
\nabla^2 g_\mu^{(1)}(X)
=
\begin{bmatrix}
\frac{n}{\mu} I_d & \mathbf{0} & \dots & \mathbf{0} \\
\mathbf{0} & \frac{n}{\mu} I_d & \dots & \mathbf{0} \\
\vdots & \vdots & \ddots & \vdots \\
\mathbf{0} & \mathbf{0} & \dots & \frac{n}{\mu} I_d
\end{bmatrix}
= \frac{n}{\mu} I_{kd}.    
\end{equation}
All eigenvalues of this block diagonal matrix are equal to $\frac{n}{\mu}$, so $g_\mu^{(1)}$ is strongly convex with parameter $\frac{n}{\mu}$.

\medskip
\noindent\textbf{Step 2: The Hessian matrix of $g_\mu^{(2)}$.}

We now rewrite the fusion term $g_\mu^{(2)}$ in matrix form. Recall the Laplacian matrix of the complete graph $K_k$:
\[
L_G = k I_k - \mathbf{1}\mathbf{1}^\top
=
\begin{bmatrix}
k-1 & -1 & \cdots & -1 \\
-1 & k-1 & \cdots & -1 \\
\vdots & \vdots & \ddots & \vdots \\
-1 & -1 & \cdots & k-1
\end{bmatrix}
\in \R^{k\times k},
\]
where $\mathbf{1} \in \R^k$ is the vector of ones.
Next, we show that
\begin{equation}\label{eq:fusion-trace}
\sum_{1\le s<t\le k} \|x_s - x_t\|^2 = \mathrm{tr}(X^\top L_G X).
\end{equation}
First, we compute $L_G X$ explicitly. Using the block structure,
\[
L_G X
=
\begin{bmatrix}
k-1 & -1 & \cdots & -1 \\
-1 & k-1 & \cdots & -1 \\
\vdots & \vdots & \ddots & \vdots \\
-1 & -1 & \cdots & k-1
\end{bmatrix}
\begin{bmatrix}
x_1^\top  \\
 x_2^\top  \\
\vdots  \\
 x_k^\top 
\end{bmatrix}
=
\begin{bmatrix}
 y_1^\top  \\
 y_2^\top  \\
\vdots \\
 y_k^\top 
\end{bmatrix},
\]
where each row $y_s^\top$ is given by
\[
y_s = (k-1)x_s - \sum_{t\neq s} x_t, \quad s=1,\dots,k.
\]
Then we obtain
$X^\top L_G X \in \R^{d\times d}$
and
\begin{align*}
\mathrm{tr}(X^\top L_G X)
= \sum_{s=1}^k x_s^\top y_s
&= \sum_{s=1}^k x_s^\top\left((k-1)x_s - \sum_{t\neq s} x_t\right)\\ 
&= (k-1)\sum_{s=1}^k \|x_s\|^2 - \sum_{s=1}^k \sum_{t\neq s} \langle x_s, x_t\rangle.
\end{align*}

On the other hand, by expanding the sum of pairwise squared distances, we get 
\begin{align*}
\sum_{1\le s<t\le k} \|x_s - x_t\|^2
&= \sum_{1\le s<t\le k} \big(\|x_s\|^2 + \|x_t\|^2 - 2\langle x_s, x_t\rangle\big) \\
&= \sum_{1\le s<t\le k} \|x_s\|^2 + \sum_{1\le s<t\le k} \|x_t\|^2 - 2\sum_{1\le s<t\le k} \langle x_s, x_t\rangle.
\end{align*}
Since each $\|x_s\|^2$ appears exactly $k-1$ times in the first two sums, we have
\[
\sum_{1\le s<t\le k} \|x_s\|^2 + \sum_{1\le s<t\le k} \|x_t\|^2
= (k-1)\sum_{s=1}^k\|x_s\|^2.
\]
Moreover,
\[
2\sum_{1\le s<t\le k} \langle x_s, x_t\rangle
= \sum_{s\neq t} \langle x_s, x_t\rangle.
\]
Therefore,
\[
\sum_{1\le s<t\le k} \|x_s - x_t\|^2
= (k-1)\sum_{s=1}^k\|x_s\|^2 - \sum_{s\neq t}\langle x_s, x_t\rangle.
\]
So we obtain
\[
\sum_{1\le s<t\le k} \|x_s - x_t\|^2 = \mathrm{tr}(X^\top L_G X),
\]
which proves \eqref{eq:fusion-trace}, and thus
\[
g_\mu^{(2)}(X) = \frac{\lambda n}{2} \sum_{1\le s<t\le k} \|x_s - x_t\|^2
= \frac{\lambda n}{2} \mathrm{tr}(X^\top L_G X).
\]
Now, the Hessian of $g_\mu^{(2)}$ can be written in block form as
\begin{equation}\label{hessian2}
\nabla^2 g_\mu^{(2)}(X)
= n\lambda \,(L_G \otimes I_d)
= n\lambda
\begin{bmatrix}
(k-1)I_d & -I_d & \dots & -I_d \\
-I_d & (k-1)I_d & \dots & -I_d \\
\vdots & \vdots & \ddots & \vdots \\
-I_d & -I_d & \dots & (k-1)I_d
\end{bmatrix}.
\end{equation}
It is well known that for the complete graph $K_k$, the Laplacian $L_G$ is positive semidefinite with eigenvalues
\[
\{0, k, k, \dots, k\}.
\]
The eigenvalues of the Kronecker product $L_G \otimes I_d$ are all eigenvalues of $L_G$ repeated $d$ times. Hence, the eigenvalues of $\nabla^2 g_\mu^{(2)}(X) = n\lambda (L_G \otimes I_d)$ are
\[
\{0, n\lambda k, n\lambda k, \dots, n\lambda k\},
\]
so $\nabla^2 g_\mu^{(2)}(X)$ is positive semidefinite.

\medskip
\noindent\textbf{Step 3: The full Hessian of $g_\mu$.}
Combining the two Hessian matrices \eqref{hessian1} and \eqref{hessian2}, we obtain
\begin{equation}\label{gmustrong}
\nabla^2 g_\mu(X)
= \nabla^2 g_\mu^{(1)}(X) + \nabla^2 g_\mu^{(2)}(X)
= \frac{n}{\mu} I_{kd} + n\lambda (L_G \otimes I_d),    
\end{equation}
which proves statement (a).

\noindent (b) It follows from \eqref{gmustrong} that
$$\nabla^2 g_\mu(X)
= \frac{n}{\mu} I_{kd} + n\lambda (L_G \otimes I_d)
\succeq \frac{n}{\mu} I_{kd}.$$
By Remark \ref{stronglyconvex}(a),  $g_\mu$ is $\gamma$-strongly convex with $\gamma = \frac{n}{\mu}$ in the sense of the Loewner ordering. 
Moreover, it follows from Remark \ref{stronglyconvex}(b) that for all $X, X' \in \R^{k\times d}$, 
\[
g_\mu(X') \ge g_\mu(X) + \langle \nabla g_\mu(X), X'-X \rangle + \frac{n}{2\mu}\|X'-X\|_F^2,
\]
which is exactly statement (b).

\medskip
\noindent (c) The gradient $\nabla g_\mu$ is an affine mapping (specifically, it is linear in its quadratic components). In the space of matrices $\mathbb{R}^{k \times d}$ equipped with the Frobenius norm, the Jacobian of the gradient corresponds to the constant Hessian matrix $\nabla^2 g_\mu$. It is well-known that the Lipschitz constant $L$ of the gradient $\nabla g_\mu$ is given by the spectral norm of its Hessian matrix (see, e.g., \cite{boyd2004convex}), namely, 
$$L = \|\nabla^2 g_\mu\|_2 = \lambda_{\max}(\nabla^2 g_\mu),$$
where $\lambda_{\max}(\nabla^2 g_\mu)$ denotes the largest eigenvalue of the Hessian.
Recall from Statement (a) that the Hessian possesses a structured block form:
\begin{equation}\label{eq:hessian-structure}
\nabla^2 g_\mu = \frac{n}{\mu} I_{kd} + n\lambda (L_G \otimes I_d).
\end{equation}
By substituting the explicit form of the complete graph Laplacian $L_G = k I_k - \mathbf{1}\mathbf{1}^\top$, the second term in \eqref{eq:hessian-structure} reveals the following block-matrix structure:
\begin{equation*}
n\lambda (L_G \otimes I_d) = n\lambda 
\begin{bmatrix}
(k-1)I_d & -I_d & \dots & -I_d \\
-I_d & (k-1)I_d & \dots & -I_d \\
\vdots & \vdots & \ddots & \vdots \\
-I_d & -I_d & \dots & (k-1)I_d
\end{bmatrix} \in \mathbb{R}^{kd \times kd}.
\end{equation*}
To compute $\lambda_{\max}(\nabla^2 g_\mu)$, we utilize the properties of the Kronecker product. The eigenvalues of $L_G \otimes I_d$ are exactly the eigenvalues of $L_G$, each repeated $d$ times. Since $L_G$ is the Laplacian of a complete graph $K_k$, its spectrum is known to be $\sigma(L_G) = \{0, k, k, \dots, k\}$. Consequently, the largest eigenvalue of the Hessian is
\begin{align*}
\lambda_{\max}(\nabla^2 g_\mu) &= \lambda_{\max} \left( \frac{n}{\mu} I_{kd} + n\lambda (L_G \otimes I_d) \right) \\
&= \frac{n}{\mu} + n\lambda \cdot \lambda_{\max}(L_G) \\
&= \frac{n}{\mu} + \lambda n k.
\end{align*}
Thus, we obtain $L = \frac{n}{\mu} + \lambda n k$ and for all $X, X' \in \mathbb{R}^{k \times d}$, the following Lipschitz inequality holds
\begin{equation*}
\|\nabla g_\mu(X') - \nabla g_\mu(X)\|_F \le \left( \frac{n}{\mu} + \lambda n k \right) \|X' - X\|_F,
\end{equation*}
which completes the proof of Statement (c).
\end{proof}

The following proposition summarizes basic convexity and regularity properties of the smoothed function $h_\mu.$
\begin{proposition}
\label{pro:hconvex-special}
The function $h_\mu$ is convex and locally Lipschitz continuous on $\mathbb{R}^{k\times d}$.  
\end{proposition} 

\begin{proof}
By the definition of the smoothed model \eqref{hmu}, the function $h_\mu$ is decomposed as $$h_\mu(X) = h_\mu^{(1)}(X) + h_\mu^{(2)}(X),$$ 
where
\[ h_\mu^{(1)}(X) = \frac{1}{2\mu} \sum_{\ell=1}^k \sum_{i=1}^n d^2(x_\ell - a_i; F^\circ) \]
and
\[ h^{(2)}(X) = \sum_{i=1}^n \max_{r=1,\dots,k} \left\{ \sum_{\ell=1, \ell \neq r}^k \rho(x_\ell - a_i) \right\}. \]
Therefore, it is not difficult to verify that $h_\mu$ is convex. Moreover, since $h_\mu$ is finite, it is locally Lipschitz continuous (see, e.g., \cite[Corollary~1.63]{mordukhovich2023easy}).
\end{proof}

We now provide the main convergence result for the sequence generated by the proposed DCA scheme.

\begin{theorem}[Monotone Descent and Stationarity] \label{thm:dca-convergence-special}
Let $\{X^{(t)}\}_{t\ge0}$ be the sequence of iterates generated by Algorithm~\ref{alg:DCA}
for the DC decomposition $f_\mu = g_\mu - h_\mu$. Then:
\begin{enumerate}
\item[{\rm (a)}] The sequence $\{X^{(t)}\}$ is bounded and satisfies
\[
f_\mu(X^{(t)}) - f_\mu(X^{(t+1)}) \;\ge\; \frac{n}{2\mu} \,\|X^{(t+1)} - X^{(t)}\|_F^2
\quad\text{for all } t\ge 0.
\]
\item[{\rm (b)}] The sequence of objective values $\{f_\mu(X^{(t)})\}$ is nonincreasing and converges to a finite value.
\item[{\rm (c)}] Every accumulation point $X^\star$ of $\{X^{(t)}\}$ is a critical point of $f_\mu$, i.e.,
\[
\nabla g_\mu(X^\star) \in \partial h_\mu(X^\star).
\]
\end{enumerate}
\end{theorem}

\begin{proof}
At each step $t$, $X^{(t+1)}$ is obtained by solving the linear system \eqref{eq:linear-system},
which is the optimality condition for the subproblem
\[
\min_{X\in\R^{k\times d}} \{ g_\mu(X) - \langle Y^{(t)}, X \rangle \},
\qquad Y^{(t)} \in \partial h_\mu(X^{(t)}).
\]
It follows from Theorem \ref{thm:smoothness-special} (a) that $g_\mu$ is $(n/\mu)$-strongly convex
and
\[
g_\mu(X) \;\ge\; g_\mu(Z)
+ \langle \nabla g_\mu(Z), X-Z \rangle
+ \frac{n}{2\mu}\,\|X-Z\|_F^2
\quad\text{for all }X,Z\in\R^{k\times d}.
\]
Applying this inequality with $X=X^{(t)}$ and $Z=X^{(t+1)}$ and using the optimality condition
\eqref{condition}, we obtain
\[
g_\mu(X^{(t)}) \;\ge\;
g_\mu(X^{(t+1)}) + \langle Y^{(t)}, X^{(t)} - X^{(t+1)} \rangle
+ \frac{n}{2\mu}\,\|X^{(t)} - X^{(t+1)}\|_F^2.
\]
On the other hand, by convexity of $h_\mu$ and $Y^{(t)} \in \partial h_\mu(X^{(t)})$, we have
\[
h_\mu(X^{(t+1)}) \;\ge\;
h_\mu(X^{(t)}) + \langle Y^{(t)}, X^{(t+1)} - X^{(t)} \rangle.
\]
Adding these inequalities yields
\[
f_\mu(X^{(t)}) - f_\mu(X^{(t+1)})
\;\ge\;
\frac{n}{2\mu}\,\|X^{(t+1)} - X^{(t)}\|_F^2,
\]
which proves the descent inequality in (a). In particular, $\{f_\mu(X^{(t)})\}$ is nonincreasing.

From the representation of $\rho_\mu$ in \eqref{eq:rho-smooth-general} and the compactness of $F^\circ$,
there exist constants $c_1>0$ and $c_2\in\mathbb{R}$ such that
\[
\rho_\mu(z)\ge c_1\|z\|-c_2
\qquad \text{for all } z\in\mathbb{R}^d.
\]
Using the same arguments as in the proof of Theorem \ref{thm1}, it follows that $f_\mu$ is coercive and hence all its sublevel sets
$$\{X\in \R^d \mid f_\mu(X) \le f_\mu(X^{(0)})\}$$ are bounded. Since $\{X^{(t)}\}$ remains in the initial
sublevel set by monotonicity, the sequence $\{X^{(t)}\}$ is bounded. Furthermore, $f_\mu$ is
bounded below, so $\{f_\mu(X^{(t)})\}$ converges to a finite value, which proves (b).

Let $X^\star$ be an accumulation point of $\{X^{(t)}\}$, and let $\{t_j\}$ be such that
$X^{(t_j)} \to X^\star$. Summing the inequality in (a) over $t$ shows that
$$\sum_t \|X^{(t+1)} - X^{(t)}\|_F^2 < \infty,$$ 
which implies that
$$\|X^{(t_j+1)} - X^{(t_j)}\|_F \to 0.$$ 
Thus, $X^{(t_j+1)} \to X^\star$ as well.
By Theorem \ref{thm:smoothness-special}(c), $\nabla g_\mu$ is continuous, and hence we have
\[
Y^{(t_j)} = \nabla g_\mu(X^{(t_j+1)}) \longrightarrow \nabla g_\mu(X^\star).
\]
Since $Y^{(t_j)} \in \partial h_\mu(X^{(t_j)})$ and the graph of $\partial h_\mu$ is closed (see, e.g.,\cite[Theorem 24.4]{rockafellar1970convex}),
we obtain
\[
\nabla g_\mu(X^\star) \in \partial h_\mu(X^\star),
\]
which shows that $X^\star$ is a critical point of $f_\mu$ and proves (c).
\end{proof}

\section{Numerical Experiments}
\label{sec:section8}

This section presents numerical experiments designed to evaluate the behavior of LDCA‑K across a range of synthetic and real datasets exhibiting diverse geometric properties. All computations were performed on a MacBook Air equipped with an Apple M4 processor and 16\,GB of unified memory. The implementation was written in Python using standard numerical libraries for linear algebra and optimization.

\subsection{Test Datasets}

We consider four datasets—one real and three synthetic—chosen to represent increasing levels of geometric complexity and cluster ambiguity. The ground-truth number of clusters is denoted by $k^\star$. Across four datasets—three synthetic with known ground truth and the
Fisher--Iris dataset—the method reliably recovers the correct cluster count in
all settings. These results demonstrate that the fusion mechanism provides effective and automatic model complexity control with minimal tuning.

\paragraph{Fisher Iris dataset.}
The Iris dataset consists of $n=150$ observations in $\mathbb{R}^4$, corresponding to sepal and petal measurements from three species. The ground-truth number of clusters is $k^\star=3$. One class (\emph{setosa}) is well separated, while the remaining two (\emph{versicolor} and \emph{virginica}) exhibit substantial overlap, making this a classical benchmark for testing clustering algorithms under weak separation.

\paragraph{Synthetic Laplace data 1 ($k^\star=3$).}
This dataset comprises three well-separated clusters in $\mathbb{R}^2$, each generated from a Laplace distribution with identical scale parameters. Each cluster contains $150$ observations. The geometry is symmetric and strongly separated, aligning closely with the noise model assumed by LDCA-K.

\paragraph{Synthetic Laplace data 2 ($k^\star=4$).}
This dataset extends the previous construction to four Laplace-distributed clusters in $\mathbb{R}^2$, each containing $100$ observations. All clusters remain well separated and balanced, but the additional cluster introduces more complex interactions between prototype fusion and deletion dynamics.

\paragraph{Synthetic Gaussian data ($k^\star=4$).}
This dataset consists of four Gaussian clusters in $\mathbb{R}^2$, with 200 samples per cluster and moderate inter‑cluster separation. To increase difficulty, the fourth cluster is placed at the centroid of the remaining three, thereby creating overlapping attraction basins. Compared to the Laplace datasets, this setting is intentionally adversarial for LDCA‑K, which—by design—draws prototypes toward their mean and can therefore obscure the correct model order when clusters are partially co-located.

\subsection{Empirical Behavior and Practical Parameter Selection}

A central challenge in prototype--based clustering models with both fusion
($\lambda$) and smoothing ($\mu$) penalties is choosing regularization
strengths that yield stable, interpretable solutions. Our aim here is
not to optimize hyperparameters exhaustively, but to document a simple,
repeatable procedure that behaves consistently across heterogeneous datasets.
We report what was done, what was observed, and where the evidence should be
interpreted with caution.

\paragraph{Regularization path construction (warm starts).}
For each dataset, we initialized the model with a modestly over--specified set
of prototypes ($k_0 = 10$). We then constructed a paired, one--dimensional
trajectory of regularization parameters
\[
\bigl\{(\lambda_t,\mu_t)\bigr\}_{t=1}^{N},\qquad
\lambda_t \in \mathrm{GeomSpace}\!\left(10^{-2},\,2.0,\,N\right),\quad
\mu_t \in \mathrm{GeomSpace}\!\left(2.0,\,10^{-4},\,N\right),
\]
with $N=100$. The path is traversed sequentially so that $\lambda_t$
monotonically increases while $\mu_t$ monotonically decreases. Critically, we
use \emph{warm starts}: the surviving prototypes obtained at step $t$ serve as
the initialization for step $t{+}1$. This produces a coherent, piecewise
continuous trajectory of solutions that reflects how the model simplifies the
prototype representation as fusion pressure strengthens (larger $\lambda$) and
smoothing weakens (smaller $\mu$).

\paragraph{Per--step procedure and logged quantities.}
At each $(\lambda_t,\mu_t)$ along the path we:
(i) run the DCA solver to convergence starting from the warm--started
prototypes; (ii) prune empty prototypes; and (iii) record
summary statistics of the resulting partition: the number of surviving
prototypes, and the minimum, maximum, mean, and standard deviation of cluster
sizes after nearest--prototype assignment. This yields a complete ``regularization
path'' of prototype counts and size statistics as functions of $t$ (and thus
implicitly of $(\lambda,\mu)$).

\paragraph{Observed behavior.}
Across all four datasets we observed a consistent pattern along the warm--started
path:
\begin{enumerate}
    \item {\em Rapid reduction from $k_0$ to the data--driven regime.}
    After a short initial segment at small $\lambda$ (large $\mu$), the number
    of prototypes drops quickly from the over--specified $k_0{=}10$ to a
    smaller, stable range.

    \item {\em Stabilization at the known cluster count.}
    On the three synthetic datasets, the path stabilizes at the ground--truth
    number of clusters ($k{=}3$ or $k{=}4$) and remains there over the vast
    majority of the trajectory. On \emph{Iris}, the path stabilizes at
    $k{=}3$, consistent with the three species.

    \item {\em Interpretable size statistics.}
    Once stabilized, the cluster--size statistics remain steady and
    dataset--appropriate: nearly uniform on balanced synthetic data; mildly
    imbalanced but stable on \emph{Iris}, reflecting the known overlap between
    \emph{versicolor} and \emph{virginica}. As $\lambda$ grows further (with
    $\mu$ shrinking), boundaries may shift and size variability can increase,
    yet the prototype count typically remains constant within a broad region.

    \item {\em Continuity induced by warm starts.}
    Warm starts yield piecewise continuous changes in both prototype locations
    and counts, avoiding erratic jumps that can arise from cold--starting each
    grid point independently. This facilitates interpretation of the entire
    path as a data--adaptive simplification process.
\end{enumerate}

\paragraph{Interpretation.}
The empirical evidence suggests that a simple paired geometric schedule for
$(\lambda,\mu)$, when combined with warm starts, yields a robust and
interpretable regularization path. In all tested cases, the method quickly
discovers a plausible cluster structure and maintains it across a wide portion
of the path. The stability of the prototype count, together with steady
cluster--size statistics, indicates qualitative robustness in practice.

\paragraph{Caveats and scope.}
Warm starts introduce path dependence: the final solution at a given
$(\lambda,\mu)$ can depend on the trajectory taken through parameter space.
Our goal here is not to claim trajectory independence or universal optimality,
but to document a practical, reproducible procedure that behaves consistently on
diverse datasets. Broader studies would help to delineate the limits and generality of these observations. Within the scope tested, however, the warm--started
path offers a clear and stable view of how fusion and smoothing jointly govern
the emergence of cluster structure.

\begin{figure}[H]
    \centering
    \includegraphics[width=0.750\textwidth]{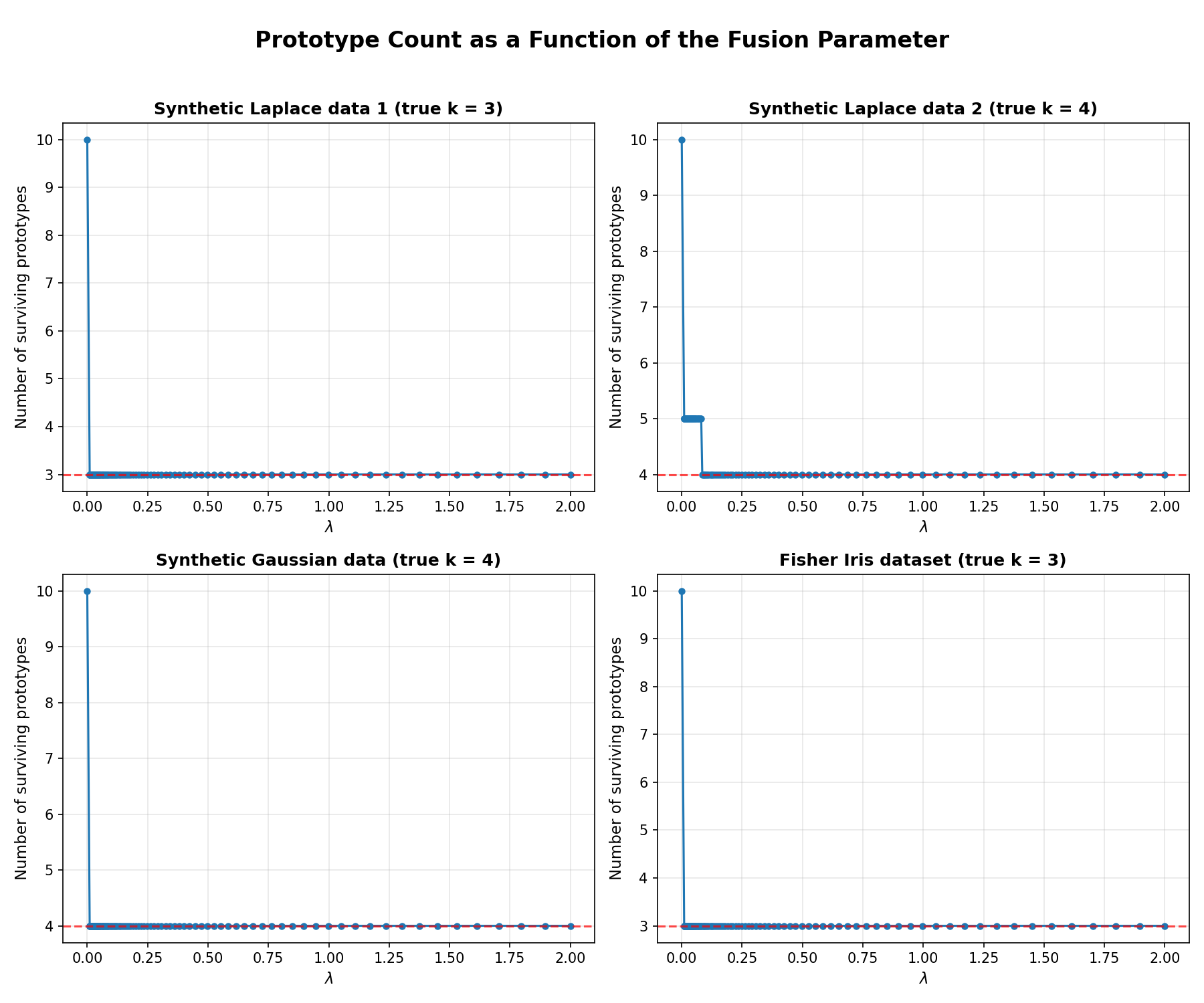}
    \caption{
Prototype count along the warm--started regularization path for four datasets.
Each panel shows the number of surviving prototypes as a function of the fusion
parameter $\lambda$, while the smoothing parameter $\mu$
simultaneously decreases along a paired geometric schedule.  All trajectories
begin with an intentionally over--specified initialization ($k_0=10$
prototypes).  As $\lambda$ increases, redundant prototypes rapidly merge, and
the model stabilizes at a data--driven number of clusters.  On the three
synthetic datasets (top row and bottom--left), the path consistently converges
to the known ground--truth structure ($k=3$ or $k=4$), and this value remains
stable across the entire range of $(\lambda,\mu)$.  On the Iris dataset
(bottom--right), the path similarly stabilizes at $k=3$, matching the three
species.  The figure illustrates the qualitative robustness of the warm--started
regularization path: once a coherent cluster structure emerges, it persists over
a broad region of the fusion parameter, even as smoothing weakens and cluster
boundaries shift.
}
    \label{fig:model-illustration}
\end{figure}

\paragraph{Synthetic Laplace data ($k^\star=3$).}

We first examine the idealized Laplace setting to illustrate the geometric behavior of LDCA-K. The dataset consists of three clusters centered at
\[
(-3,0), \quad (3,0), \quad (0,\sqrt{27}),
\]
each corrupted by independent Laplace noise with scale $0.25$. Each cluster contains $150$ observations, for a total of $n=450$ points. 

Figure~\ref{fig:dca_traj} shows prototype trajectories during a single inner DCA optimization prior to deletion. Prototypes are initially attracted to distinct regions of the data, but the smoothed DC geometry ultimately induces collapse onto shared stationary points. Redundant prototypes are subsequently removed by the outer deletion mechanism.

\begin{figure}[H]
\centering
\includegraphics[width=0.40\linewidth]{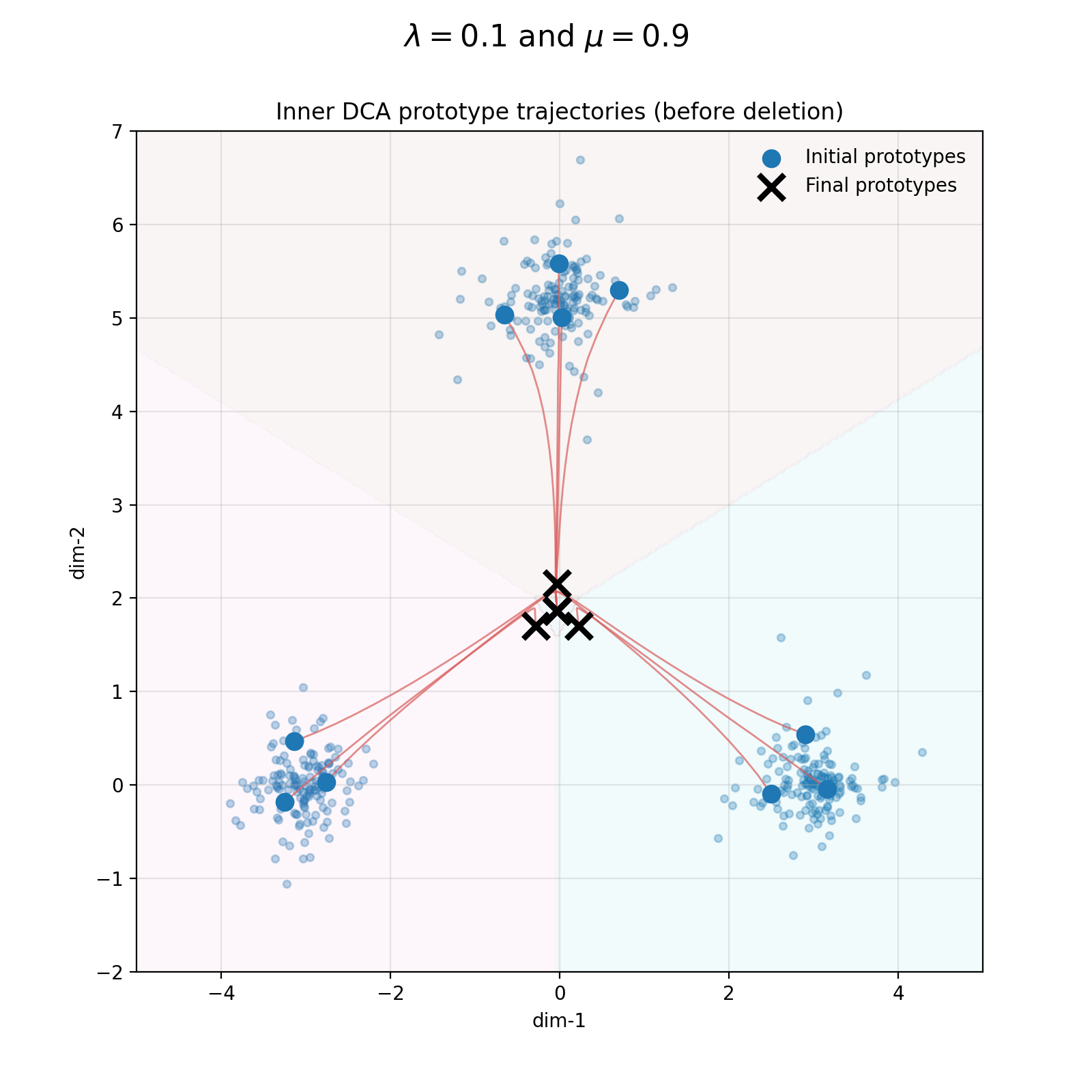}
\caption{Prototype trajectories during inner DCA optimization prior to deletion. Multiple prototypes collapse onto shared stationary points.}
\label{fig:dca_traj}
\end{figure}

\begin{figure}[H]
\centering
\includegraphics[width=0.40\linewidth]{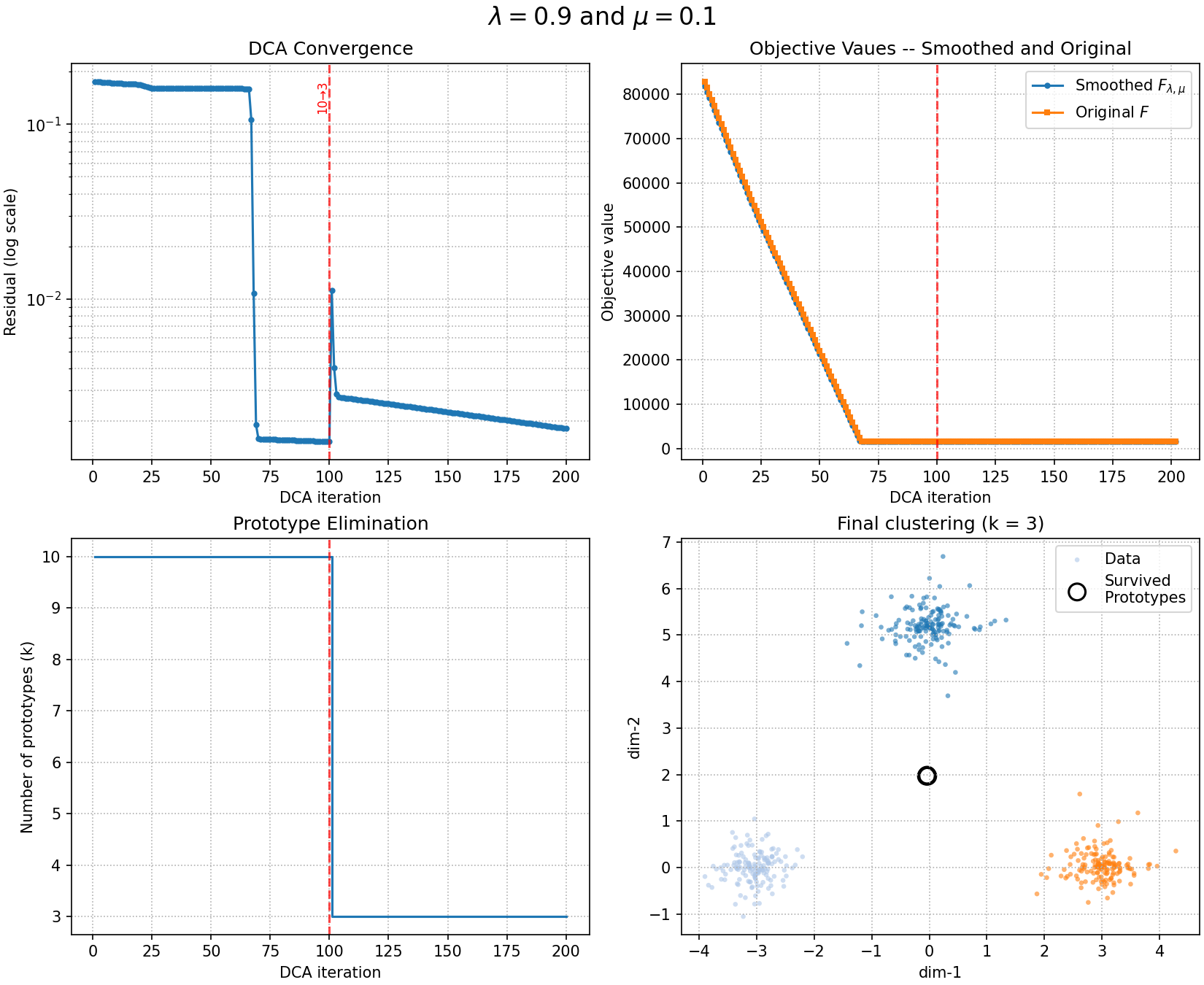}
\caption{Representative LDCA-K run illustrating convergence behavior and correct model selection despite prototype fusion.}
\label{fig:example_conv_plots}
\end{figure}

To quantify clustering accuracy under prototype fusion, we compute the Adjusted Rand Index (ARI) as a function of the total pairwise prototype spread
\[
D_{\text{centers}} = \sum_{\ell < j} \|x^\ell - x^j\|.
\]
Table~\ref{tab:fusion_vs_ari_K03} shows that ARI remains near $1.0$ even when $D_{\text{centers}}$ indicates near-complete prototype fusion, confirming that LDCA-K prioritizes partition recovery over centroid identifiability.

\begin{table}[H]
\centering
\caption{Clustering quality as a function of prototype separation for the $k^\star=3$ Laplace dataset.}
\label{tab:fusion_vs_ari_K03}
\begin{tabular}{l c c c c c}
\toprule
Center distance & Mean $k_{\text{eff}}$ & Std.\ $k_{\text{eff}}$ & Mean ARI & Std.\ ARI & Count \\
\midrule
$< 0.01$      & 3.00 & 0.00 & 1.00 & 0.00 & 43 \\
$[0.01,1.0)$  & 3.00 & 0.00 & 1.00 & 0.00 & 6787 \\
$[1.0,2.0)$   & 3.00 & 0.00 & 1.00 & 0.00 & 1270 \\
$[2.0,5.0)$   & 3.00 & 0.00 & 1.00 & 0.00 & 100 \\
$[5.0,10.0)$  & 3.00 & 0.00 & 1.00 & 0.00 & 456 \\
$[10.0,25.0)$ & 3.00 & 0.00 & 1.00 & 0.00 & 1310 \\
$\ge 25.0$    & 4.00 & 0.00 & 0.996 & 0.001 & 34 \\
\bottomrule
\end{tabular}
\end{table}

\paragraph{Synthetic Laplace data ($k^\star=4$).}

We next consider a four-cluster Laplace dataset with centers at
\[
(0,0),\ (2,0),\ (0,2),\ (2,2),
\]
and $100$ observations per cluster. LDCA-K is initialized with $k_0=10$ prototypes.

Figure~\ref{fig:phase_diagram} shows the resulting phase diagram over the $(\lambda,\mu)$ grid. For small $\mu$, excessive fusion can lead to underestimation of $k^\star$. As $\mu$ increases, smoothing stabilizes the DC updates and produces a broad region in which the correct model order is consistently recovered.

\begin{figure}[H]
\centering
\includegraphics[width=0.45\linewidth]{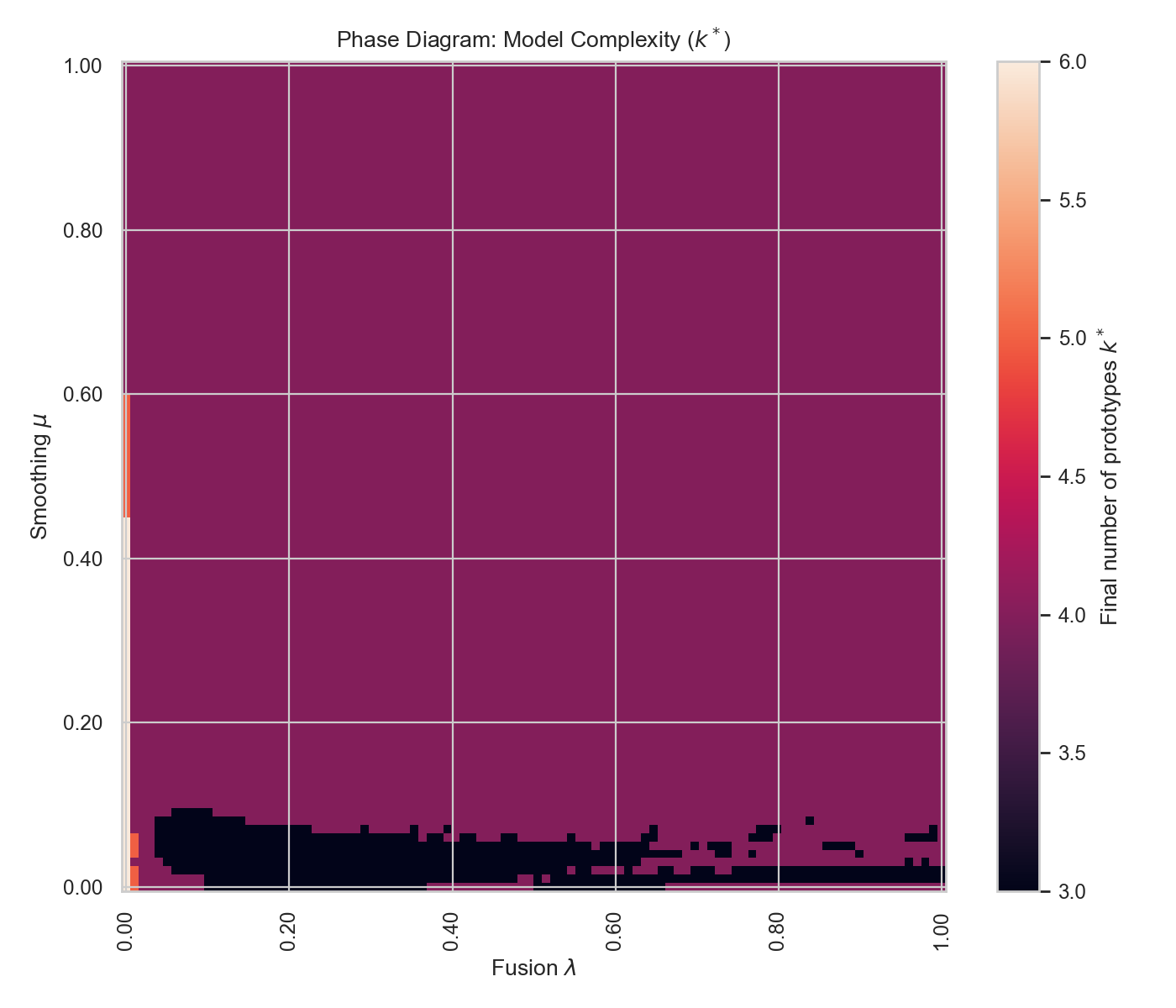}
\caption{Phase diagram of the effective number of clusters over the $(\lambda,\mu)$ grid.}
\label{fig:phase_diagram}
\end{figure}

Representative solutions for low and high $\mu$ are shown in
Figure~\ref{fig:example_mu_comparison}, illustrating the complementary roles
of fusion and smoothing.

\begin{figure}[H]
\centering
\includegraphics[width=0.48\linewidth]{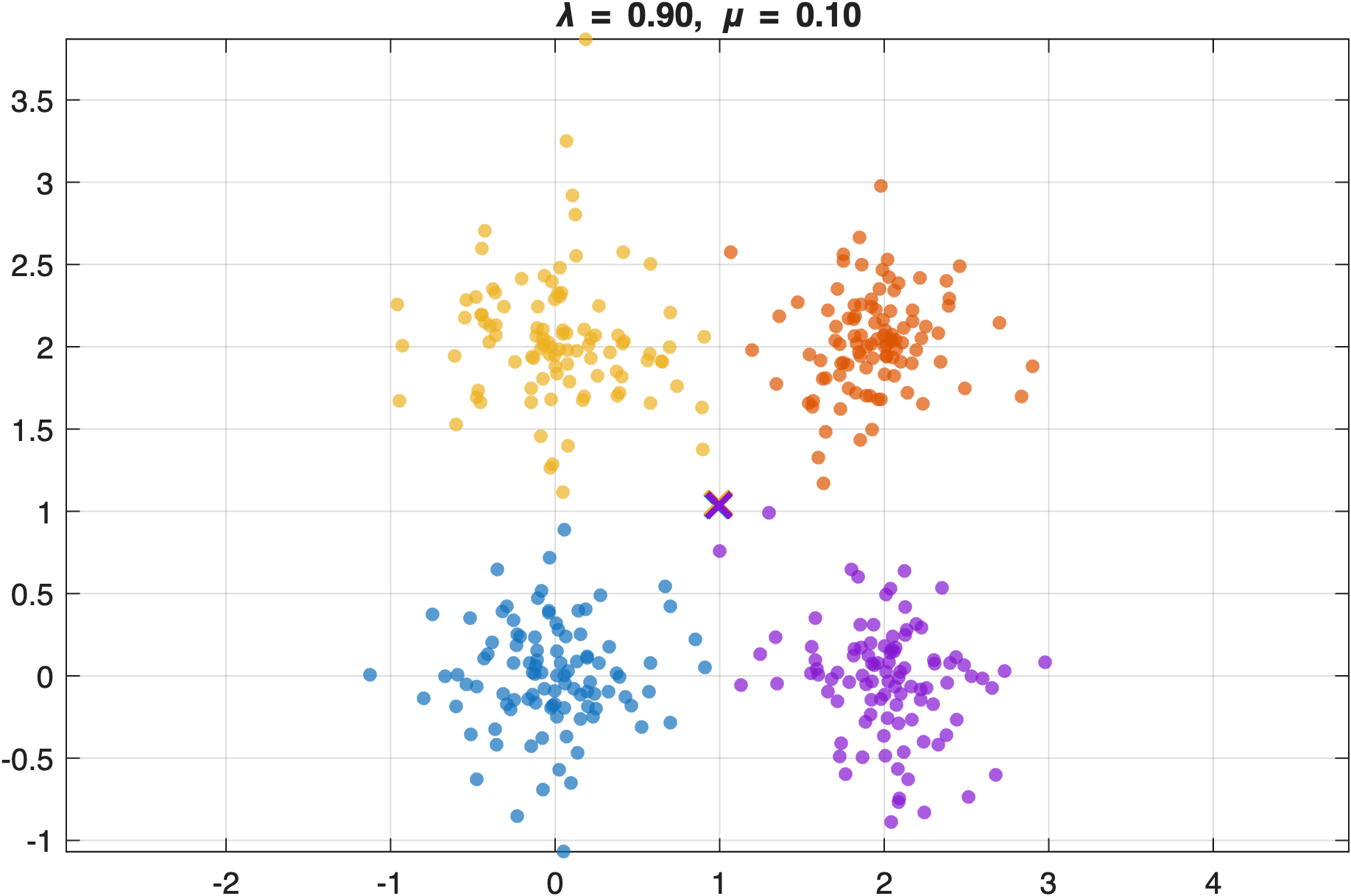}\hfill
\includegraphics[width=0.48\linewidth]{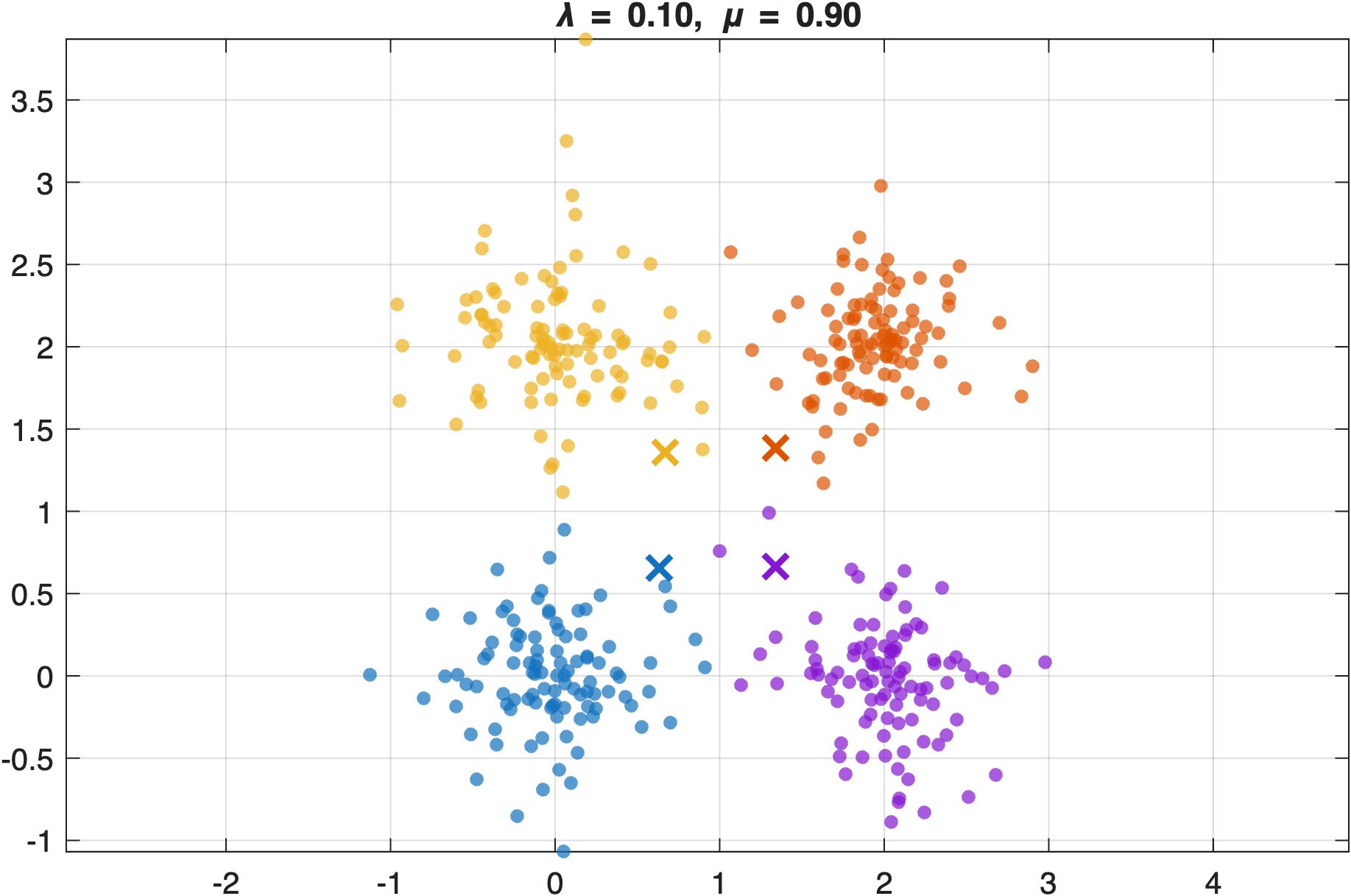}
\caption{
Representative LDCA-K solutions illustrating the complementary roles of fusion
and smoothing.
\textbf{Left:} $\lambda=0.90$, $\mu=0.10$, where strong fusion with limited smoothing
leads to aggressive prototype merging and a reduced effective model.
\textbf{Right:} $\lambda=0.10$, $\mu=0.90$, where increased smoothing stabilizes the
optimization and yields four well-separated prototypes aligned with the true
cluster centers.
}
\label{fig:example_mu_comparison}
\end{figure}

\paragraph{Synthetic Gaussian data ($k^\star=4$).}

In the most ambiguous setting, one Gaussian cluster is centered at the global mean of the other three. In this case, LDCA-K recovers the true number of clusters about 18\% of parameter settings (see Table~\ref{tab:fusion_vs_ari_K03}).

\begin{figure}[H]
\centering
\includegraphics[width=0.65\linewidth]{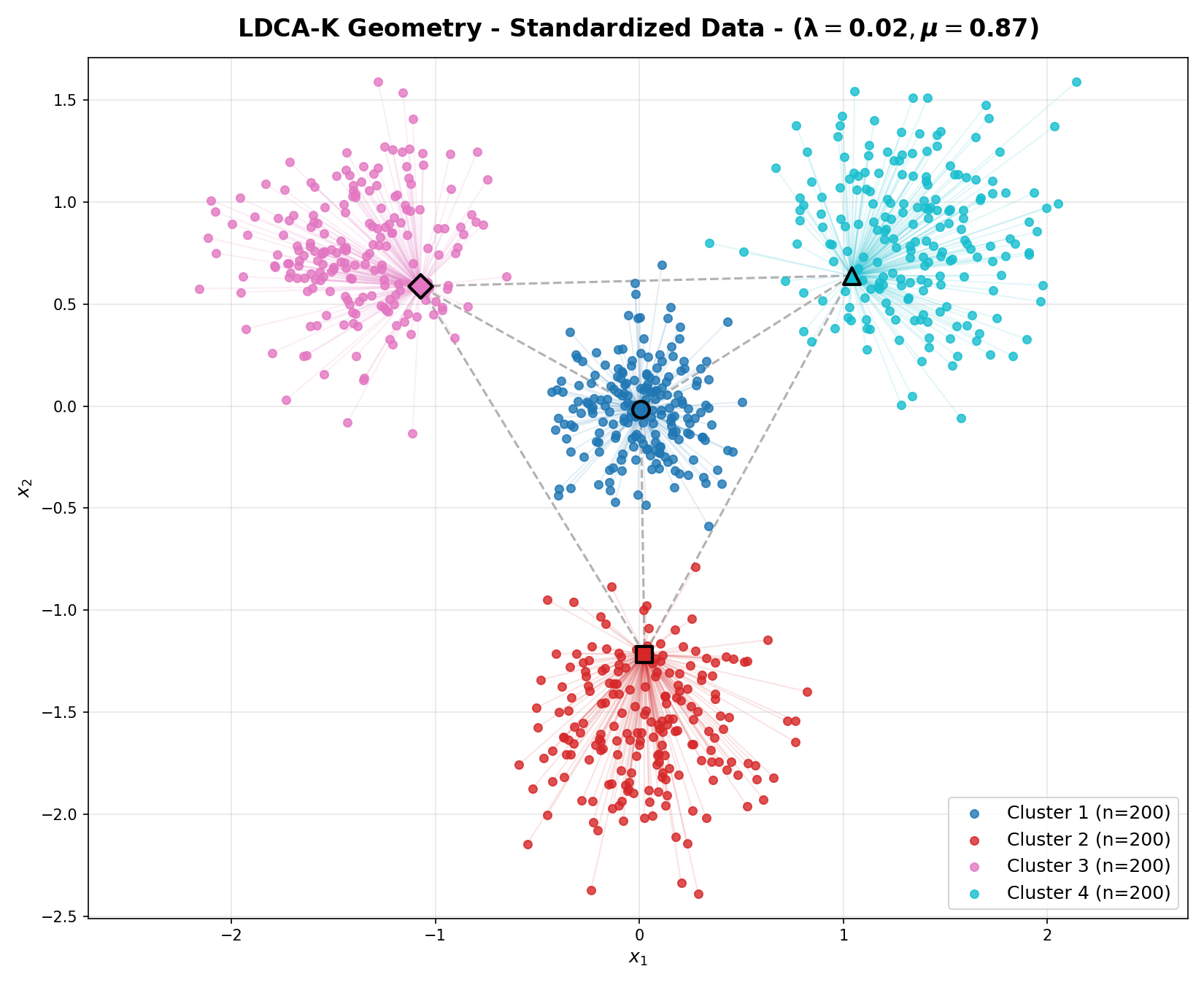}
\caption{
Final Euclidean cluster assignments for standardized synthetic Gaussian data ($k^\star=4$) dataset.} 
\label{fig:final-clusters}
\end{figure}

\begin{figure}[H]
\centering
\includegraphics[width=0.65\linewidth]{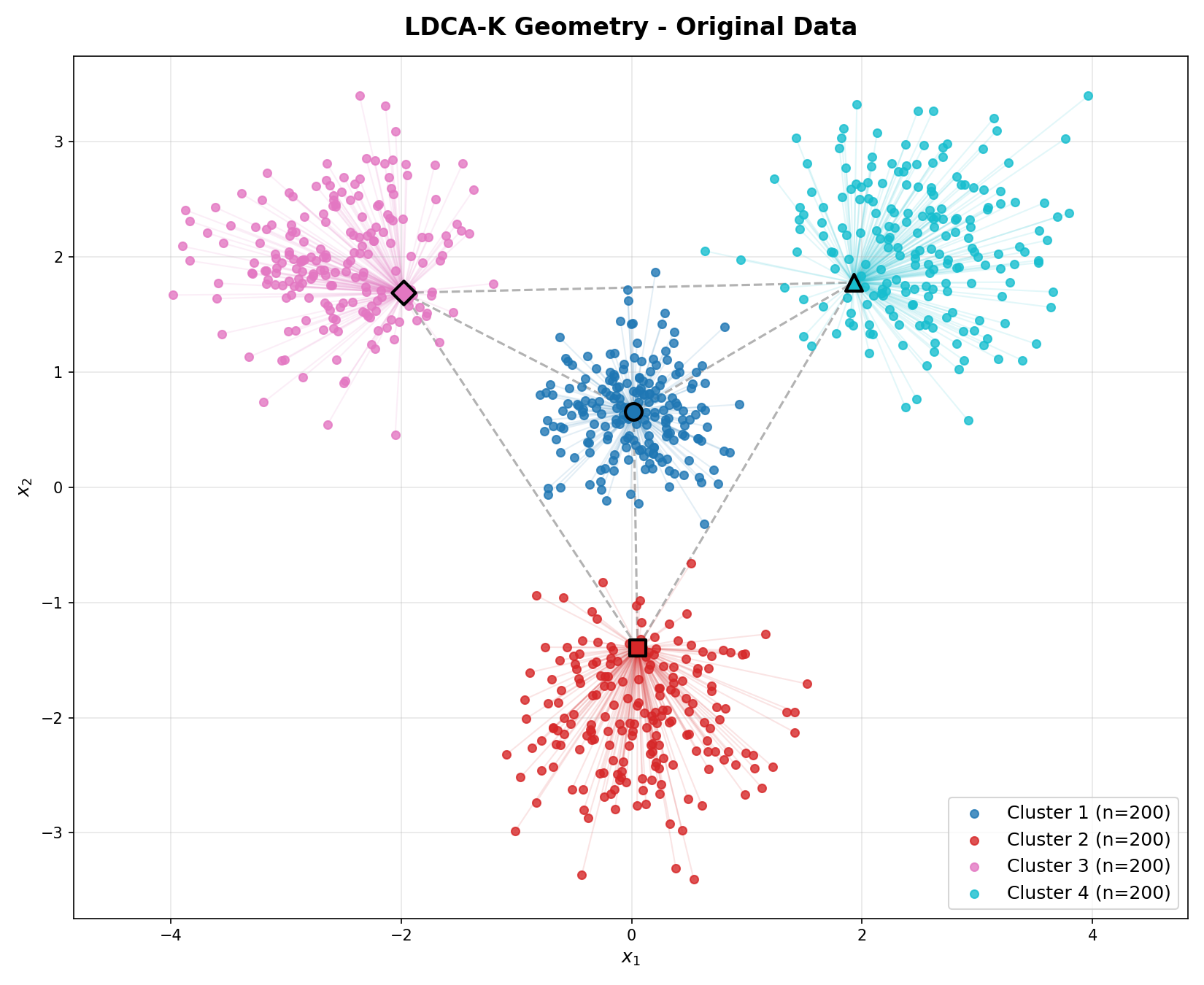}
\caption{
The corresponding Euclidean cluster assignments in the original dataset.} 
\label{fig:final-clusters1}
\end{figure}

\begin{figure}[H]
\centering
\includegraphics[width=0.75\linewidth]{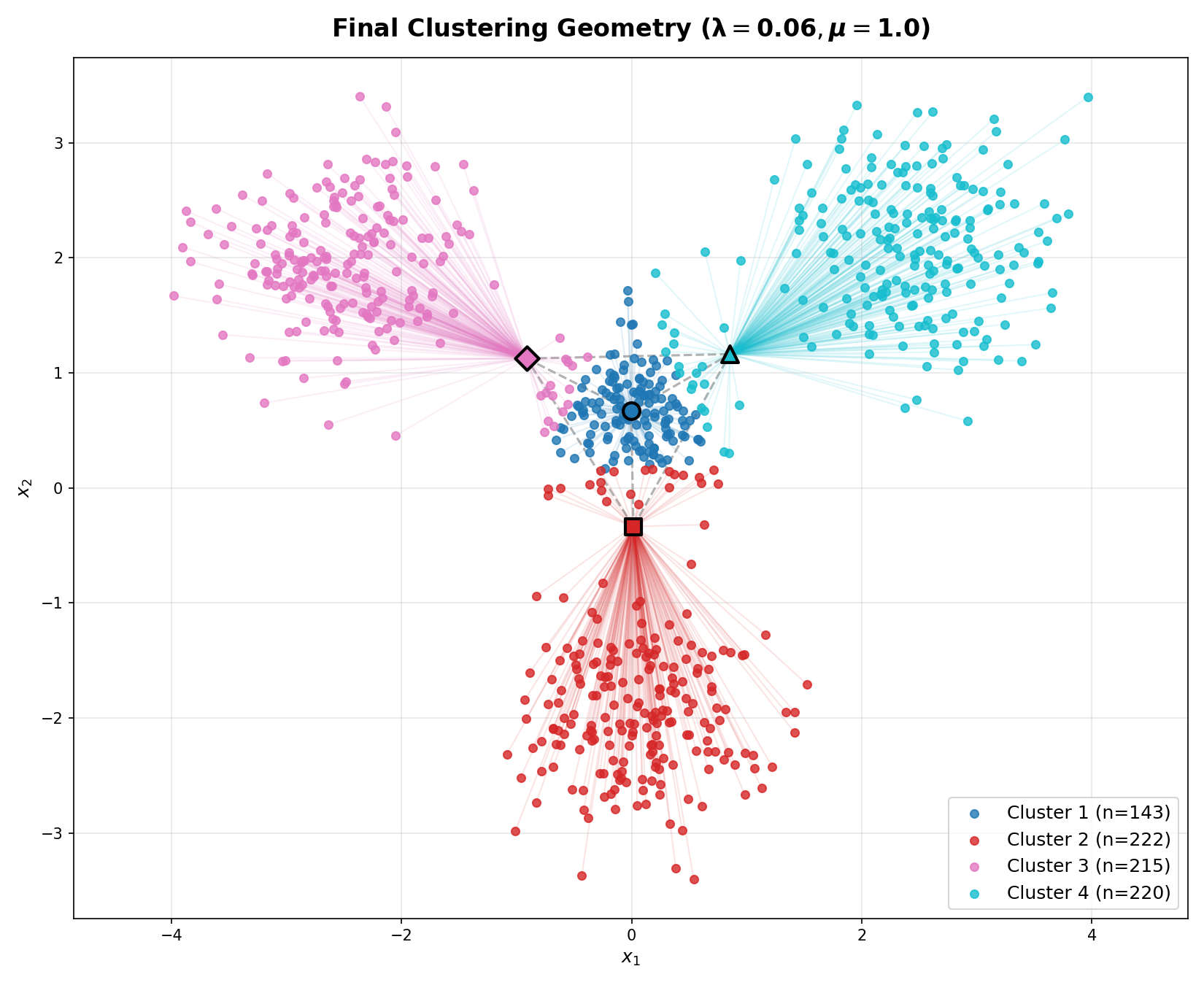}
\caption{
Final Euclidean cluster assignments for best run on the original dataset.} 
\label{fig:final-clusters2}
\end{figure}

\subsection{Discussion}

The empirical results indicate that LDCA--K provides a coherent and effective
framework for simultaneous clustering and model selection across a range of
geometric conditions. Beginning from intentionally overparameterized prototype
sets, the method quickly removes redundant components through the combined
action of fusion and pruning, yielding compact cluster representations without
requiring delicate initialization or manual tuning. Prototype fusion is not an
auxiliary post--processing step but an intrinsic property of the objective: as
the fusion penalty grows, closely positioned prototypes naturally coalesce,
thereby performing model order reduction in a data--driven manner.

In datasets with clear separation, LDCA--K reliably identifies the ground--truth
number of clusters and maintains this structure across a wide region of the
regularization path. This behavior is particularly notable given the use of
warm starts: despite the inherent path dependence, the number of surviving
prototypes stabilizes early and remains unchanged even as the regularization
parameters continue to vary. Such stability suggests that the method is
capturing persistent geometric structure rather than transient artifacts of the
optimization trajectory.

In more ambiguous settings---for example, when clusters overlap or exhibit
anisotropic shapes---the method displays a principled sensitivity to the
regularization parameters. Instead of producing fragmented or unstable
solutions, LDCA--K tends to favor simpler and more interpretable clusterings
that are consistent with the underlying data geometry. This reflects a desirable
property: the effective number of clusters emerges from a balance between the
fusion pressure, the smoothing penalty, and the intrinsic separation in the
data, rather than from brittle thresholding or hand-tuned heuristics.

Overall, the experiments demonstrate that the number of clusters recovered by
LDCA--K is driven primarily by geometric structure and the regularization path,
not by initialization randomness or numerical instability. While a broader set
of datasets would be needed to characterize the full scope of this behavior,
the evidence presented here supports the view that LDCA--K offers a stable,
interpretable, and practically robust mechanism for clustering with automatic
model complexity control.

\end{document}